\documentclass[a4paper,oneside,11pt]{article}

\usepackage[utf8]{inputenc}
\usepackage[T1]{fontenc}
\usepackage{textcomp}
\usepackage[english]{babel}
\usepackage[autolanguage]{numprint}
\usepackage{hyperref}
\usepackage{graphicx}
\usepackage{verbatim}

\usepackage{amsmath,amssymb,amsthm}

\renewcommand\epsilon\varepsilon
\renewcommand\phi\varphi
\renewcommand\geq\geqslant
\renewcommand\leq\leqslant
\renewcommand\ln\log

\newcommand\RR{\mathbb{R}}

\newcommand\ab\allowbreak
\newcommand{\ent}{\mathrm{Ent}}
\newcommand{\var}{\mathrm{Var}}
\newcommand{\tQ}{\widetilde{Q}}

\newtheorem{The}[equation]{Theorem}
\newtheorem{Lem}[equation]{Lemma}
\newtheorem{Pro}[equation]{Proposition}
\newtheorem{Cor}[equation]{Corollary}
\newtheorem{Def}[equation]{Definition}
\newtheorem{Rem}[equation]{Remark}
\newtheorem{Claim}[equation]{Claim}

\numberwithin{equation}{section}

\title{Hamilton-Jacobi equations on graph and applications}
\author{Yan SHU\footnote{Mod\'elisation al\'eatoire de Paris Ouest Nanterre La D\'efense(MODAL'X), Email: yshu@u-paris10.fr}}
\date{\today}

\begin{document}
\maketitle
\begin{abstract}
This paper introduces a notion of gradient and an infimal-convolution operator that extend properties of solutions of Hamilton Jacobi equations to more general spaces, in particular to graphs. As a main application, the hypercontractivity of this class of infimal-convolution operators is connected to some discrete version of the log-Sobolev inequality and to a discrete version of Talagrand's transport inequality.
\end{abstract}
key words: Hamilton-Jacobi equations; Weak-transport entropy inequalities; Modified Log-Sob inequalities on graphs
\section{Introduction}
The following Hamilton-Jacobi initial value problem
\begin{equation}\label{e}
\begin{cases}
\frac{\partial v(x,t)}{\partial t}+\frac{1}{2}|\nabla_x v(x,t)|_x^2=0 & (x,t) \in
 M \times(0,\infty)\\v(x,0)=f(x) & x \in M,
\end{cases}
\end{equation}
where $(M,g)$ is a smooth Riemannian manifold and $|\,\cdot \,|_x$ is the norm on $T_xM$ associated to the metric $g$ at point $x$, together with its explicit solution, given by the celebrated Hopf-Lax formula,
\begin{equation} \label{hl}
Q_tf(x)=\inf_{y\in M}\left\{f(y)+\frac{1}{2t}d(x,y)^2\right\}, \qquad \quad  t>0, x\in M
\end{equation}
where $d$ denotes the geodesic distance on $M$ (with \textit{e.g.}\ $f \colon M \to \RR$ Lipschitz) are very classical and have a lot of applications in Analysis, Physics and Probability Theory (let us mention applications in large deviations theory, statistical mechanics, mean field games, optimal control, optimal transport, functional inequalities, they also have deep connections with geometry (Ricci curvature) etc.). We refer to the books by Evans \cite{evans}, Barbu and Da Prato \cite{barbu}  and Villani \cite{Villani} for an introduction and for related topics.
\medskip

An important effort has been made recently to generalize such a classical theory to more general situations, for example by replacing the Riemannian manifold $M$ by a general metric space
(see \textit{e.g.} \cite{ambrosio,GRS14}). We refer to the introduction of \cite{gangbo} for a review of the literature and in particular on the various notions of viscosity solution introduced in the metric spaces setting. One non trivial issue is to give a proper definition of gradient in order for Equation  \eqref{e} to make sense, and, with that respect, an important ingredient is that the space needs to be continuous. In particular, the known theories fail to directly generalize to discrete structures such as graphs.
\medskip

The aim of the present paper is precisely to introduce a notion of gradient and to use an inf-convolution operator that extend, in some sense, \eqref{e} and \eqref{hl}, to graphs, with a specific focus for applications on  functional inequalities. It turns out that our approach, originally devised to deal with the graph setting, works also for general metric spaces.
\medskip

We introduce now the notion of gradient and the inf-convolution operator we shall deal with through the paper.  Let $(X,d)$ be a complete, separable metric space such that balls are compact.
\medskip

The (length of the) gradient we shall consider is defined as
$$
|\widetilde{\nabla}f|(x):=\sup_{y \in X}\frac{[f(y)-f(x)]_-}{d(x,y)}
$$
where $[a]_-=\max(0,-a)$ is the negative part of $a\in \RR$ (by convention $0/0=0$). We observe that, in discrete setting, one usually deals with quantity involving $|f(y)-f(x)|$, with $y$ a neighbour of $x$ (a property we denote by $x \sim y$), which is usually less than $|\widetilde{\nabla}f|(x)$. However, if $f$ is assumed to be a convex function, then $|\widetilde{\nabla}f|(x)=\sup_{y \sim x}[f(y)-f(x)]_-$. Also, in $\RR^n$ equipped with the usual Euclidean distance, if $f$ is a smooth convex function, $|\widetilde{\nabla}f|$ coincides with the usual length of the gradient $|\nabla f|(x)=\sqrt{\sum_i \partial_i f^2}(x)$ (and it always holds $|\widetilde{\nabla}f| \geq |\nabla f|$).
\medskip

As for the inf-convolution operator, we observe that there is at least one important difference with respect to the continuous setting. Indeed, as we shall explain in detail later, under very mild assumptions, there is no hope of finding a family of mappings $(D_t)_{t > 0}$ such that  $Q_tf(x):=\inf_{y\in V}\{f(y)+D_t(x,y)\}$
(where $x,y$ belong to the vertex set $V$ of a graph $G=(V,E)$) satisfies the usual semi-group property
$Q_{t+s}=Q_t( Q_s)$.
\medskip

To overcome this problem, we may use the following weak inf-convolution operator,
$$
\tQ_tf(x)=\inf_{p\in \mathcal{P}(X)} \left\{ \int  f\,dp + \frac{1}{2t}\left(\int  d(x,y)\,p(dy)\right)^2\right\},\qquad t>0,
$$
defined for all bounded measurable functions $f$, where $\mathcal{P}(X)$ denotes the set of Borel probability measures on $X$. This weak inf-convolution operator is naturally linked (via some variant of the Kantorovich duality theorem proved in \cite{GRST14}) to the following weak optimal transport-cost introduced by Marton \cite{Mar96b}:
\begin{equation} \label{weakt}
\widetilde T_2(\nu|\mu) := \inf \left\{ \int  \left( \int  d(x,y)\,p_x(dy)\right)^2\,\mu(dx) \right\},
\end{equation}
where $\mu, \nu$ are probability measures on $X$ and where the infimum is running over all couplings $\pi(dx,dy)=p_x(dy)\mu(dx)$ of $\mu$ and $\nu$ (\textit{i.e.}\ $\pi$ is a probability measure on $X \times X$ with first marginal $\mu$ and second marginal $\nu$ and $(p_x)_{x \in X}$ denotes the regular conditional probability of the second marginal knowing the first). Note that integrals stand for sums in the discrete setting.
Such a transport-cost appeared in the literature as an intermediate tool to obtain concentration results, see Marton \cite{Mar86,Mar96a,Mar96b}, Dembo \cite{Dem97}, Samson \cite{Sam00,Sam03,Sam07}, Wintemberger \cite{Win13}, and as a discrete counterpart of the usual $W_2$-Kantorovitch-Wasserstein distance in some
displacement convexity property of the entropy along interpolating paths on graphs, see Gozlan-Roberto-Samson-Tetali \cite{GRST12,GRST14}.
\medskip

Our main theorem is the following counterpart of \eqref{e}.

\begin{The}\label{thm:main}
Let $f \colon X \rightarrow \RR$ be a lower semi-continuous function bounded from below.
Then, for all $x \in X$, it holds
$$
\begin{cases}
\frac{\partial}{\partial t}\tQ_tf(x)+\frac{1}{2} |\widetilde{\nabla}\tQ_tf|^2(x) \leq 0 & \qquad \forall t>0 \\
\frac{\partial}{\partial t}\tQ_tf(x)|_{t=0}+\frac{1}{2} |\widetilde{\nabla}f|^2(x) = 0
& \qquad t=0 .
\end{cases}
$$
\end{The}

With such a result in hand, we can then follow the work by Bobkov, Gentil and Ledoux \cite{BGL} to prove a  result analogous to the celebrated Otto and Villani Theorem \cite{OV00}. Namely we shall prove that some log-Sobolev type inequality is equivalent to an hypercontractivity property of the semi-group $\widetilde Q_t$, which in turn, by a duality argument due to Gozlan et al.\  \cite{GRST14}, implies some Talagrand type transport-entropy inequality. To state this result one needs to introduce some additional notations.
Consider the usual $q$-norm of a function $g$ on $X$ defined by $\|g\|_q = (\int |g|^q \,d\mu )^{1/q}$, $q \in \mathbb{R}$,  with, when this makes sense, $\|g\|_0:=\lim_{q \to 0} \|g\|_q=\exp\{\int \log g \,d\mu\}$, and when $g\geq0$, consider also the entropy functional defined by $\mathrm{Ent}_{\mu}(g)=\int g \log g \,d\mu - \int g \,d\mu \log \int g \,d\mu$.

\begin{Cor} \label{cormain}
Let $\mu$ be a probability measure on $X$ and $C>0$. Then
\begin{description}
\item $(i)$ If for all bounded measurable function $f \colon X \to \mathbb{R}$ it holds,
\begin{equation}\label{eqMLS-(2/c)intro}
\mathrm{Ent}_{\mu} (e^f)\leq C \int  |\widetilde{\nabla}f|^2e^{f} \,d\mu,
\end{equation}
 then for every $\rho \geq 0$, every $t\geq 0$ and every bounded measurable function $f$,
 \begin{equation}\label{eq7intro}
\|e^{\tQ_tf}\|_{\rho +\frac{2t}{C}}\leq \|e^f\|_\rho.
\end{equation}
Conversely, if (\ref{eq7intro}) holds for some $\rho > 0$ and for all $t\geq 0$, then \eqref{eqMLS-(2/c)intro} holds.
\item $(ii)$
If for all bounded measurable function $f \colon X \to \mathbb{R}$ it holds,
\begin{equation}\label{eqMLS-(2/c)intro1}
\mathrm{Ent}_{\mu} (e^f)\leq C \int  |\widetilde{\nabla}(-f)|^2e^{f} \,d\mu,
\end{equation}
 then \eqref{eq7intro} holds for every $\rho \leq 0$, every $t\in [0,-\rho C/2]$ and every bounded measurable function $f$.
Conversely, if (\ref{eq7intro}) holds for some $\rho < 0$ for all $t\in [0,-\rho C/2)$, then \eqref{eqMLS-(2/c)intro1} holds.
\end{description}
\end{Cor}

\begin{The}\label{equivalencelogsobtalagrand}
Let $\mu$ be a probability measure on $X$ and $C>0$. Then the following conditions are equivalent
\begin{description}
\item $(i)$ $\mu$ satisfies the modified log-sob inequality \eqref{eqMLS-(2/c)intro} with constant $C_1>0$.
\item $(ii)$ There exists $C_2>0$ for all $\nu$ probability measure on $X$,
\begin{equation} \label{t2tilde}
\widetilde T_2(\mu|\nu) \leq C_2 H(\nu|\mu).
\end{equation}
\end{description}
where $H(\nu|\mu)$ is the relative entropy of $\nu$ with respect to $\mu$, \textit{i.e.}\
$H(\nu|\mu)=\mathrm{Ent}_{\mu}(g)$ if $\nu \ll \mu$ and $g:=d\nu/d\mu$, and $H(\nu|\mu)=+\infty$ otherwise.
Moreover, $(i)\Rightarrow (ii)$ with $C_2=C_1/2$, $(ii)\Rightarrow (i)$ with $C_1=2C_2$.
\end{The}

The log-Sobolev-type inequality \eqref{eqMLS-(2/c)intro} is implied by the usual Gross' inequality \cite{G75} in the continuous setting (since $|\widetilde{\nabla}f| \geq |\nabla f|$). In discrete, there exist a lot of different versions of the log-Sobolev inequality -- that are all equivalent in the continuous, thanks to the chain rule formula -- each of them having some nice property (connection to the decay to equilibrium of Markov processes, concentration phenomenon etc.). We refer the reader to the paper by Bobkov and Tetali \cite{BobkovTetali}
for an introduction to many of these inequalities and related properties. In particular, in
\cite{BobkovTetali}, the log-Sobolev type inequality \eqref{eqMLS-(2/c)intro} is studied, with some local gradient in place of $\widetilde{\nabla}$. As we shall prove below, the usual log-Sobolev inequality in discrete, with transitions given by a Markovian matrix, implies \eqref{eqMLS-(2/c)intro}. In turn, since such an inequality is very well studied in many situations (see \textit{e.g.}\ the monographs \cite{Saloff,1} and \cite{martinelli,guionnet} for results on general graphs and examples coming from physics)
this provides a lot of examples of non trivial measures (on graphs) that satisfy the Talagrand-type transport-entropy inequality \eqref{t2tilde}.
\medskip

Inequality \eqref{t2tilde}  is related to the concentration phenomenon and was studied by the authors listed above (Dembo, Gozlan, Marton, Roberto, Samson, Tetali, Wintenberger). However, proving directly \eqref{t2tilde} for non-trivial measures is not an easy task and, to the best of our knowledge, there exist very few examples of measures satisfying \eqref{t2tilde}. In fact, Corollary \ref{cormain} above, together with the important literature on the log-Sobolev inequality provide at once new examples.
\medskip

That \eqref{eqMLS-(2/c)intro} implies \eqref{t2tilde}(with $|\nabla|$ in place of $\widetilde \nabla$) is known, in the continuous setting, as Otto-Villani's Theorem \cite{OV00}. Such a theorem was proved using Otto calculus in the original paper
\cite{OV00} in the Riemannian setting. Soon after, Bobkov, Gentil and Ledoux \cite{BGL} gave an alternative proof based on Hamilton-Jacobi equation. Then, it was generalized to compact measured geodesic spaces by Lott and Villani \cite{LV07,LV09} (see also \cite{balogh}), and to general metric spaces by Gozlan \cite{gozlan09}, see also Gozlan, Roberto and Samson \cite{GRS14} and for an approach based on the Hamilton-Jacobi Semi-group. Later on, the original ingredients of Otto-Villani's paper were successfully adapted to the general metric space framework by Gigli and Ledoux \cite{GL13}. Our proof follows the Hamilton-Jacobi approach of \cite{BGL}.
We point out that \eqref{t2tilde} implies \eqref{eqMLS-(2/c)intro}(with $|\nabla|$ in place of $\widetilde \nabla$) is not true in the continuous setting.
\medskip

We conclude this introduction with some more comments and a short roadmap of the paper.
\medskip

In the next section, we introduce various notations and derive some technical and useful facts on the operator  $\tQ_t$ that might be of independent interests. We also prove that $Q_tf(x):=\inf_{y\in V}\{f(y)+D_t(x,y)\}$ usually does not satisfy any semi-group property. In Section \ref{sec:proof} we prove
Theorem \ref{thm:main}. Section \ref{functional inequality} is dedicated to the applications to functional inequalities, while Section \ref{sec:examples} collects some examples that will illustrate our main theorems.
Finally, in the Appendix we prove a technical result.
\medskip

We mention that the results above can be proved in a more general situation, namely by replacing
the cost $x^2/2$ by a general convex function $\alpha$ (with the Fenchel-Legendre dual function $\alpha^*$ appearing in the corresponding Hamilton-Jacobi equation), see below.
Finally we observe that there exist other papers dealing with Hamilton-Jacobi equation on graphs, but with very different perspectives (approximation scheme, viscosity solution, etc.).
We refer to \cite{camilli} and references therein for an account on these topics.

\section{Preliminaries} \label{preliminaries}

In this section, we introduce some notations and prove some properties on the operator $\tQ_t$ and on the gradient $\widetilde \nabla$ that will be useful later on.

\subsection{Notations} \label{sec:notation}

\subsubsection*{Space}
In all the paper $(X,d)$ stands for a polish space (\textit{i.e.}\ complete and separable), such that closed balls are compact. In the discrete case, $X=G=(V,E)$ will denote a (simple) connected graph with vertex set $V$ and edge set $E$ (given $(x,y) \in E$, we may write $x \sim y$). We assume that all vertices have finite degree. The graph distance will be denoted by $d$. Next, $\mathcal{P}(X)$ stands for the set of all probability measure on $X$, and, in order to emphasize the discrete character, when $X=G=(V,E)$ is a graph,
we may use instead $\mathcal{P}(V)$.


\subsubsection*{Inf-convolution operator}

Throughout the paper, $\alpha:\RR^+\rightarrow \RR^+$ denotes a convex function, of class $\mathcal{C}^1$, such that $\alpha(0)=\alpha'(0)=0$ (so that $\alpha$ is non-decreasing).
Its Fenchel-Legendre transform is denoted by $\alpha^*$ and defined by
$\alpha^*(x):=\sup_{y \in \mathbb{R}^+} \{ xy - \alpha(y)\}$, $x \in \RR^+$.
A typical example of such a function is given by $\alpha(x)=x^2/2$, and more generally by
$\alpha(x)=x^p/p$ for which $\alpha^*(x)=x^q/q$ with $p^{-1}+q^{-1}=1$, $p,q >1$. Another example (related to the Poincar\'e inequality, see Section \ref{Poincare}) is the following, called \emph{quadratic-linear cost}, $\alpha_a^h(x):=ax^2$ if $x\in [0,a]$ and $\alpha_a^h(x)=2ax-ah^2$ if $x\geq a$, with $a,h>0$ two parameters.

Given $f \colon X \to \mathbb{R}$, we denote by $\|f\|_{Lip}:=\sup_{x,y, x\neq y} \frac{f(y)-f(x)}{d(x,y)}$ the Lipschitz norm of $f$.


Next we define the (inf-convolution) operators $Q_tf$, $\tQ_t$ and $\hat{Q}_t$.
Given $f:X\rightarrow \RR$ bounded from below, $x \in X$  and $t>0$, let
\begin{equation*}
Q_tf(x):=\inf_{y\in X}\left\{f(y)+t\alpha\left(\frac{d(x,y)}{t}\right)\right\},
\end{equation*}
\begin{equation*}
\tQ_tf(x):=\inf_{p\in \mathcal{P}(X)}
\left\{\int f\,dp+t\alpha\left(\frac{\int d(x,y)\,p(dy)}{t}\right)\right\},
\end{equation*}
Restricting the infimum to the set of Dirac masses, we observe that necessarily  $Q_tf \geq \tQ_tf$. As we shall see on the example of the two points space, the latter inequalities are strict in general. However, in specific cases (if $f$ is convex and $X = \RR^n$ equipped with a norm $\| \cdot \|$) equality holds. We illustrate this in the following proposition.

\begin{Pro}\label{exconvex}
Assume that $X=\RR^n$ equipped with a distance $d$ coming from a norm $\|\cdot\|$. Then, for all $f \colon \RR^n \mapsto \RR$ convex and bounded from below, $\tQ_tf=Q_tf$.
\end{Pro}
\begin{proof}
By convexity of $f$ and of the norm, Jensen's Inequality and the monotonicity of $\alpha$ imply that, for all $p \in \mathcal{P}(\RR^n)$ such that $\int \|x\|\,p(dx)$ is finite, it holds
$$
\int f(y)\,p(dy)+t\alpha\left(\frac{\int \|x-y\|\,p(dy)}{t}\right)
\geq f\left(\int y\,p(dy)\right)+t\alpha\left(\frac{1}{t} \| x - \int y \,p(dy) \|\right) .
$$
Hence, setting $z:=\int y\,p(dy) \in \RR^n$ and optimizing we get
\begin{align*}
\tQ_tf(x)
&=
\inf_{p \text{ s.t. } \int \|x\|\,p(dx)<\infty}
\left\{\int f\,dp+t\alpha\left(\frac{\int \|x- y\|\,p(dy)}{t}\right)\right\} \\
& \geq
\inf_{z\in \RR^n}
\left\{f(z)+t\alpha\left(\frac{\|x-z\|}{t}\right)\right\}
 =Q_tf(x)
\end{align*}
which leads to the desired result.
\end{proof}

\subsection{Properties of the operator $\tQ_t$} \label{sec:qt}
In all what follows, $f:X\to \RR$ is a lower semicontinuous function bounded from below. Let
\begin{equation} \label{defm}
m_f(t,x):= \left\{ p \in \mathcal{P}(X): \tQ_tf(x)=\int f\,dp +t\alpha\left(\frac{\int d(x,y)\,p(dy)}{t}\right) \right\}
\end{equation}
be the set (possibly empty) of probability measures $p$ realizing the infimum in the definition of $\tQ f$.
The following lemma shows that this set is not empty.
\begin{Lem}\label{mfnonempty}
If $f:X\to\RR$ is lower semicontinuous and bounded from below, then $m_f(t,x)\neq \emptyset$ for all $t>0$ and $x \in X.$
\end{Lem}
We postpone the proof of the lemma at the end of the section.
\medskip

In order to state the main theorem of this section we need to introduce some additional notations. Given $x\in X$, let $I_x:=\{ d(x,y), y \in X \} \subset \RR^+$ be the image of the function $X \ni y \mapsto d(x,y)$.  Since $(X,d)$ is a polish space such that all closed balls are compact, $I_x$ is a closed subset of $\RR$.
Then, define $f_x \colon I_x \rightarrow \mathbb{R}$ as
$$
f_x(u):=\min_{y \in X: d(x,y)=u} \{f(y) \}
$$
and notice that $f_x(0)=f(x)$. We will sometime consider that $f_x$ is defined on $[0,\infty)$ by setting $f_x(u)=+\infty$ when $x$ is outside $I_x.$ Let $\widetilde{I}_x$ be the convex hull of $I_x$ (since closed balls are assumed to be compact, $\widetilde{I}_x$ is one of the following intervals $[0,\sup I_x]$ (if $I_x$ is bounded) or $[0,+\infty)$ (if $I_x$ is unbounded)). Let $\widetilde{f}_x:\RR^+ \rightarrow \RR\cup\{+\infty\}$ be the convex hull of $f_x$, that is to say the greatest convex function $g:\RR\to\RR\cup\{+\infty\}$ such that $g(u)\leq f_x(u)$ for all $u\in I_x$. The function $\widetilde{f}_x$ takes finite values on $\widetilde{I}_x$ and is $+\infty$ outside $\widetilde{I}_x$.
Another way to define $\widetilde{f}_x$ on $\widetilde{I}_x$ is given in the following lemma whose proof is postponed at the end of the section. Let $\mathcal{P}_u(I_x)$ be the set of probability measures on $I_x$ with expectation $u$, \textit{i.e.}\ $\int_{I_x} y\, p(dy) = u$.
\begin{Lem}\label{Lem:convhull}Let $f:X\to \RR$ be a lower semicontinuous function and define $f_x$ and $\widetilde{f}_x$ as above. Then, for all $u \in \widetilde{I}_x$,
\begin{equation}\label{eq:convhull}
\widetilde{f}_x(u)=\inf \left\{\int_{I_x} f_x(w)\,q(dw) : q\in \mathcal{P}_u(I_x)\text{ charging at most two points}\right\}.
\end{equation}
Moreover, the function $\widetilde{f}_x$ is continuous on $\widetilde{I}_x$ and lower semicontinuous on $\RR.$
\end{Lem}
 The following lemma illustrate when the latter infimum could be achieved. This lemma seems classical and it might be found in some convex analyses document.
 \begin{Lem}\label{lemclassic}
Let $f$ be a lower semi-continuous function bounded from below define on a close set $I\subset \RR$. Let $g$ be the largest convex function such that $g\leq f$ on $I$. Then for all affine function $h$, define $I_h:=[a,b]$ be the maximum interval such that $g-h$ reaches its minimum, if $a\neq \infty$, then $a\in I$ and $f(a)=g(a)$, the same conclusion holds for $b$ if $b\neq \infty$.
\end{Lem}
\begin{proof}
 Without loss of generality, we can suppose that $I_h=[a,b]$ with $a\neq \pm \infty$. It is enough to show that $f(a)=g(a)$, the other cases are similar.
The definition of $g$ implies directly that $g(a)\leq f(a)$, so we now turn to prove the inverse inequality.
Changing $h$ into $h+constant$, we can suppose that $g-h=0$ on $I_h$ and $g-h>0$ on $\RR\setminus I_h$. Let $h_n$ the affine function such that $h_n(a-1/n)=g(a-1/n)$ and $h_n(a+1/n)=g(a+1/n)$. By definition of $I_h$, $h_n(a-1/n)>h(a-1/n)$ and $h_n(a+1/n)\geq h(a+1/n)$. It follows that $h_n(a)>h(a)=g(a)$.
Thus, if we define $g_n:x\mapsto \max \{g(x), h_n(x)\}$, then $g_n$ is a convex function greater than $g$. Thus, the definition of $g$ implies that the existence of $z_n\in I$ such that $f(z_n)<g_n(z_n)$. Notice that $g_n= g$ on $\RR\setminus [a-1/n,a+1/n]$, so $z_n\in [a-1/n,a+1/n]$. Hence, $\lim_{n\rightarrow \infty} z_n =a$ and it holds
\begin{align*}
g(a)
&=
h(a)=\lim_{n\rightarrow \infty} h(z_n)\leq \lim_{n\rightarrow \infty} f(z_n) \\
&\leq
\lim_{n\rightarrow \infty} g_n(z_n)\leq \lim_{n\rightarrow \infty} \max \{g(a-1/n),g(a+1/n)\}=g(a)
\end{align*}
Thus, by lower semi-continuity of $f$, we have $g(a)=\lim_{n\rightarrow \infty} f(z_n)\geq f(\lim_{n\rightarrow \infty} z_n)=f(a)$.
The proof is completed.
\end{proof}
 As a consequence of the latter lemma, suppose that the largest affine part contains $(u,\widetilde{f}_x(u))$ is $([a_u,b_u],\widetilde{f}_x([a_u,b_u]))$, if $b_u<\infty$, then we have  $a_u,b_u\in I_x$ and $\widetilde{f}_x(a_u)=f_x(a_u)$, $\widetilde{f}_x(b_u)=f_x(b_u)$. Hence,
$$\widetilde{f}_x(u)=\int_{I_x} f_x(w)\,q(dw),$$
where $q=\lambda\delta_{a_u}+(1-\lambda)\delta_{b_u}$ with $\lambda$ satisfies $u=\lambda a_u+(1-\lambda b_u)$.

Finally, let
\begin{equation} \label{defmt}
\widetilde{m}_f(t,x):= \left\{ u \in \RR^+: {Q}_t\widetilde{f}_x(0)= \widetilde{f}_x(u) +t\alpha\left(\frac{u}{t}\right) \right\}.
\end{equation}
This set is easily seen to be non-empty using the lower semicontinuity of $\widetilde{f}_x$  (see also Item $(ii)$ of the following result.)

\begin{The}\label{the1}
Set $\beta(x):=x\alpha'(x)-\alpha(x)$, $x \geq 0$.
Let $f \colon X \to \RR$ be bounded from below and lower semi-continuous. Then,
\begin{description}
\item{(i)}
For all $t>0$, all $x \in X$, it holds $\tQ_t f(x) = Q_t \widetilde{f}_x(0)$;
\item{(ii)} Assume that the cost function $\alpha$ is strictly increasing, then for all $t>0$ and all $x \in X$, it holds
\begin{equation}\label{eq:mf=mftilde}
\left\{ \int d(x,y)\,p(dy) : p \in {m}_f(t,x) \right\} = \widetilde{m}_f(t,x).
\end{equation}
more generally for all cost function $\alpha$, it holds
$$\left\{ \int d(x,y)\,p(dy) : p \in {m}_f(t,x) \right\}\subset \widetilde{m}_f(t,x)$$
and
$$
\widetilde{m}_f(t,x) \subset \bigcap_{\epsilon>0}\left\{ \int d(x,y)\,p(dy) : p\in m_f^\epsilon(t,x)\right\},
$$
where $m_f^\epsilon(t,x) = \left\{p \in \mathcal{P}(X) : \int f\,dp + t\alpha\left(\frac{\int d(x,y)\,p(dy)}{t}\right) \leq Q_tf(x)+\epsilon \right\}.$\\
In particular, when $X$ is compact,  \eqref{eq:mf=mftilde} holds for all $\alpha$.
\item{(iii)}
For all $x\in X$ and all $t>0$, the function $u\mapsto \beta(u/t)$ is constant on $\widetilde{m}_f(t,x)$. In particular, the function $p\mapsto  \beta \left(\int d(x,y)\,p(dy)/t\right)$ is constant on $m_f(t,x).$
\item {(iv)} For all $t>0$, $x \in X$ and $p \in m_f(t,x)$, it holds
\begin{equation} \label{froid}
\frac{\partial}{\partial t} \tQ_tf(x)= - \beta \left( \frac{\int d(x,y)\,p(dy)}{t} \right)  ;
\end{equation}

\end{description}
\end{The}

\begin{proof}[Proof of Theorem \ref{the1}]
Let us prove Item \textit{(i)}.
Fix $f \colon X \to \RR$ bounded from below and lower semi-continuous, and $x \in X$. It holds
\begin{align*}
\tQ_tf(x)
& =
\inf_{p\in \mathcal{P}(X)}\left\{\int f\,dp+t\alpha\left(\frac{\int d(x,y)\,p(dy)}{t}\right)\right\} \\
& =
\inf_{u \in \RR^+} \left\{ g_x(u)+t\alpha\left(\frac{u}{t}\right)\right\},
\end{align*}
where
\[
g_x(u)=\inf \left\{\int f\, dp : p\in \mathcal{P}(X): \int d(x,y)\,p(dy)=u\right\},\qquad u\in \RR^+.
\]
Let us show that $g_x(u) = \widetilde{f}_x(u)$ $u\in \RR^+.$ If $u$ is outside $\widetilde{I}_x$, then both functions are equal to $+\infty$ and there is nothing to prove. Let us show that $g_x=\widetilde{f}_x$ on $\widetilde{I}_x.$ First choosing, in the definition of $g_x$, $p=\delta_y$ for some $y\in X$ such that $d(x,y)=u\in I_x$, one gets that $g_x(u) \leq f(y)$. Optimizing over all $y$ such that $d(x,y)=u$, one concludes that $g_x(u) \leq f_x(u)$ for all $u\in I_x.$
Moreover the function $g_x$ is easily seen to be convex. By definition of the convex hull of $f_x$, it follows that $g_x(u) \leq \widetilde{f}_x(u)$ for all $u\in \widetilde{I}_x.$
Now let us show that $g_x \geq \widetilde{f}_x$. For all $y\in X$, it holds $f(y) \geq f_x(d(x,y))$. Therefore, if $p$ is such that $\int d(x,y)\,p(dy)=u \in \widetilde{I}_x$, then denoting by $\widetilde{p}\in \mathcal{P}_u(I_x)$ the image of $p$ under the map $y\mapsto d(x,y)$, it holds
\begin{equation}\label{eq:pfthe1}
\int f(y)\,p(dy) \geq \int f_x(d(x,y))\,p(dy) = \int f_x(v)\,\widetilde{p}(dv) \geq \int \widetilde{f}_x(v)\,\widetilde{p}(dv) \geq \widetilde{f}_x(u),
\end{equation}
where the last inequality follows from Jensen inequality.
Optimizing over $p$, yields to $g_x \geq \widetilde{f}_x$ on $\widetilde{I}_x$ and so $g_x=\widetilde{f}_x$ and this completes the proof.
\medskip

Now, we prove Item $(ii)$. Let $p \in m_f(t,x)$ and $u=\int d(x,y)\,p(dy)$. Then, according to \eqref{eq:pfthe1}, one has $\widetilde{f}_x(u) \leq \int f\,dp$.
Hence, using the very definition of $m_f(t,x)$, Item $(i)$ and the definition of $Q_t \widetilde{f}_x(0)$, it holds
\begin{align*}
\widetilde{f}_x(u) + t \alpha \left( \frac{u}{t} \right)
& \leq
\int f\,dp + t \alpha \left( \frac{\int d(x,y)\,p(dy)}{t} \right)
 =
\tQ_t f(x)
 =
Q_t \widetilde{f}_x(0) \\
& \leq
\widetilde{f}_x(u) + t \alpha \left( \frac{u}{t} \right)
\end{align*}
It follows that $Q_t \widetilde{f}_x(0) = \widetilde{f}_x(u) + t \alpha \left( \frac{u}{t} \right)$ and thus that $u \in \widetilde{m}_f(t,x)$
which, in turn, guarantees that $\left\{ \int d(x,y)\,p(dy) : p \in m_f(t,x) \right\} \subset \widetilde{m}_f(t,x)$.

Conversely, let $u \in \widetilde{m}_f(t,x)$.
 Firstly assume that the cost function $\alpha$ is strictly increasing. If $u=0$, then it suffice to take $p=\delta_0$ and it is easy to see that $p \in m_f(t,x)$. Now suppose that $u>0$. Let $([a_u,b_u], \widetilde{f}_x([a_u,b_u]))$ be the largest affine part of the graph $\widetilde{f}_x$ which contains $(u,\widetilde{f}_x(u))$. If $b_u<\infty$, then  thanks to lemma\ref{lemclassic} $f_x(a_u)=\widetilde{f}_x(a_u)$ and $f_x(b_u)=\widetilde{f}_x(b_u)$. As a consequence, there exist $y_1$ and $y_2$ such that $f_x(a_u)=f(y_1)$ and $f_x(b_u)=f(y_2)$, $d(x,y_1)=a_u$, $d(x,y_2)=b_u$. It is suffice to define $p:=\lambda \delta_{y_1}+(1-\lambda)\delta_{y_2}$ where $\lambda$ satisfies $\lambda a_u+(1-\lambda) b_u=u$. Moreover, by Item $(i)$ and by definition of $\widetilde{m}_f(t,x)$ we have
$$
\tQ_t f(x)
 =
Q_t \widetilde{f}_x(0)
=
\widetilde{f}_x(u) + t \alpha \left( \frac{u}{t} \right)
=
\int f\,dp+ t \alpha \left( \frac{\int d(x,y)\,p(dy)}{t} \right) \geq \tQ f(x)
$$
which proves that $p \in m_f(t,x)$ and thus that $u \in \left\{ \int d(x,y)\,p(dy) : p \in m_f(t,x)\right\}$.

Now we turn to the case $b_u=\infty$.  Let $h$ be the affine function which is coincide with $\widetilde{f}_x$ on $[a_u, \infty)$. Since $\widetilde{f}_x$ is bounded from below, so is $h$. It follows that $h'\geq 0$. Hence, $z\mapsto \widetilde {f}_x(z)+t\alpha(z/t)$ is strictly increasing on $[a_u,\infty)$. On the other hand, $u\in \widetilde{m}_f(t,x)$ implies that $u$ achieves the minimum of function $z\mapsto \widetilde {f}_x(z)+t\alpha(z/t)$. Thus $u=a_u$ and there exists $y\in X$ such that $d(x,y)=u$ and $f(y)=f_x(u)= \widetilde{f}_x(u)$ by lemma \ref{lemclassic}. Again by Item $(i)$ and by definition of $\widetilde{m}_f(t,x)$ we deduce that the probability $p:=\delta_y \in m_f(t,x)$ and $u\in \left\{ \int d(x,y)\,p(dy) : p \in m_f(t,x)\right\}$.

Now we turn to prove the general case:
According to \eqref{eq:convhull}, for all $\epsilon>0$, there exists $q^\epsilon \in \mathcal{P}_u(I_x)$ charging at most two points such that $\int f_x(v)\,q^\epsilon(dv) \leq \widetilde{f}_x(u) + \epsilon$. For any $v$ in the support of $q^\epsilon$, there exists $y_v \in X$ such that $d(x,y_v)=v$ and $f(y_v) = f_x(v)$ (here we use the facts that $f$ is lower-semicontinuous and balls are compact). Define $p^\epsilon=\sum_{v \in \mathrm{Supp}(q^\epsilon)} q^\epsilon(\{v\})\delta_{y_v}$. By construction, it holds $\int d(x,y)\,p^\epsilon(dy) = \int v\,q^\epsilon(dv)=u$ and $\int f(y)\,p^\epsilon(dy) = \int f_x(v)\,q^\epsilon(dv).$ Moreover, by Item $(i)$ and by definition of $\widetilde{m}_f(t,x)$ we have
$$
\tQ_t f(x)
 =
Q_t \widetilde{f}_x(0)
=
\widetilde{f}_x(u) + t \alpha \left( \frac{u}{t} \right)
=
\int f\,dp^\epsilon + t \alpha \left( \frac{\int d(x,y)\,p^\epsilon(dy)}{t} \right)-\epsilon \geq \tQ f(x)-\epsilon
$$
which proves that $p \in m_f^\epsilon(t,x)$ and thus that $u \in \left\{ \int d(x,y)\,p(dy) : p \in m_f^\epsilon(t,x) \right\}$.
So it holds $\widetilde{m}_f(t,x) \subset \bigcap_{\epsilon>0} \{\int d(x,y)\,p(dy) : p \in m^\epsilon_f(t,x)\}:=A(t,x).$

Now, let us assume that $X$ is compact, and let us show that the set $A(t,x)=\{\int d(x,y)\,p(dy) : p \in m_f(t,x)\}$. Let $u \in A(t,x)$ and $\epsilon_n $ be a sequence of positive numbers tending to $0$ ; then there exists a sequence $p_n \in m_f^{\epsilon_n}(t,x)$ such that $u=\int d(x,y)\,p_n(dy)$. According to Prokhorov Theorem, $\mathcal{P}(X)$ is compact, therefore one can assume without loss of generality that $p_n$ converges weakly to some $p^*$. Since $X$ is compact, the function $y\mapsto d(x,y)$ is bounded and continuous and therefore the functional $p\mapsto \int d(x,y)\,p(dy)$ is continuous. One concludes that $\int d(x,y)\,p^*(dy) =u$. Now let us show that $p^* \in m_f(t,x)$. Since $f$ is lower semicontinuous $\liminf_{n\to \infty}\int f\,dp_n \geq \int f\,dp^*.$ Since $p_n \in m_f^{\epsilon_n}(t,x)$, letting $n\to\infty,$ one concludes that $\int f\,dp^* + t\alpha\left(\frac{\int d(x,y)\,p^*(dy)}{t}\right) \leq \widetilde{Q}_tf(x)$ and so $p^* \in m_f(t,x).$
This ends the proof of Item $(ii)$.
\medskip

Let us prove Item $(iii)$. By definition, $\widetilde{m}_f(t,x)$ is the set where the convex function $F(v)=\widetilde{f}_x(v) + t\alpha (v/t)$ attains its minimum on $\RR^+.$ Therefore $\widetilde{m}_f(t,x)$ is an interval. Suppose that $u_1 < u_2$ are in $\widetilde{m}_f(t,x)$, then $F$ is constant on $[u_1,u_2]$. Since both functions $\widetilde{f}_x$ and  $t\alpha (\,\cdot\,/t)$ are convex, this easily implies that these functions $\widetilde{f}_x$ and $t\alpha (\,\cdot\,/t)$ are both affine on $[u_1,u_2].$ In particular, $\alpha'(u/t)$ is constant on $[u_1,u_2].$ It follows that $\beta(u_2/t) = (u_2/t)\alpha'(u_2/t)-\alpha(u_2/t) = (u_2/t)\alpha'(u_1/t) - \alpha(u_1/t) - \alpha'(u_1/t)(u_2-u_1)/t = \beta(u_1/t)$. This shows that $\beta(\,\cdot\,/t)$ is constant on $\widetilde{m}_f(t,x).$
\medskip

Let us turn to the proof of Item $(iv)$. According to \cite[Theorem 1.10]{GRS14} (which applies since $\widetilde{f}_x : \RR \to \RR\cup\{+\infty\}$ is bounded from below and, according to Lemma \ref{Lem:convhull}, lower-semicontinuous), it holds
\[
\frac{dQ_t \widetilde{f}_x(0)}{dt_+} = - \beta \left(\frac{\max \widetilde{m}_f(t,x)}{t}\right),
\]
and
\[
\frac{dQ_t \widetilde{f}_x(0)}{dt_-} = - \beta \left(\frac{\min \widetilde{m}_f(t,x)}{t}\right),
\]
where $d/dt_\pm$ stands for the right and left derivatives.
According to Item $(iii)$ the function $\beta(\,\cdot\, / t)$ is constant on $\widetilde{m}_f(t,x)$. Therefore, the left and the right derivatives of $t\mapsto Q_t \widetilde{f}_x(0)$ are equal, and so the function is actually differentiable in $t$.
According to Item $(i)$, $\tQ_tf(x)=Q_t\widetilde{f}_x(0)$ and, according to Item $(ii)$, $\{\int d(x,y)\,p(dy) : p \in m_f(t,x)\}\subset \widetilde{m}_f(t,x)$ which proves \eqref{froid}.
\end{proof}
Let us mention an interesting consequence of the proof of Item $(ii)$. Let us denote by $\mathcal{P}_2(X)$ the set of probability measures on $X$ charging at most two points:
$$
\mathcal{P}_2(X) := \left\{ (1-s) \delta_x + s \delta_y : \; s \in [0,1], \; x,y \in X \right\} .
$$
\begin{Pro}
Let $f:X\to \RR$ be a lower semicontinuous function bounded from below. Then
\[
\tQ_tf(x) = \inf\left\{ \int f\,dp + t\alpha\left(\frac{\int d(x,y)\,p(dy)}{t}\right) : p \in \mathcal{P}_2(X)\right\}.
\]
\end{Pro}
\proof
It is enough to show that for all $\epsilon>0$, $m_f^\epsilon(t,x)\cap \mathcal{P}_2(X)\neq \emptyset$ (recall the definition of $m^\epsilon(t,x)$ given in Item $(ii)$ of Theorem \ref{the1}). Actually, this follows immediately from the argument given in the proof of Item $(ii)$. Indeed, we showed there that for all $u \in \widetilde{m}_f(t,x)$ there exists $p \in \mathcal{P}_2(X)\cap m^\epsilon_f(t,x)$ such that $\int d(x,y)\,p(dy) = u.$
\endproof

Now let us prove Lemmas \ref{mfnonempty} and \ref{Lem:convhull}.

\begin{proof}[Proof of Lemma \ref{mfnonempty}]
Since $f$ is lower semicontinuous and bounded from below, the function $p \mapsto \int f\,dp$ is lower semicontinuous with respect to the weak convergence topology of $\mathcal{P}(X)$. For the same reason $p\mapsto \int d(x,y)\,dp$ is also lower semicontinuous. Therefore, the function $F(p)= \int f\,dp +t\alpha\left(\frac{\int d(x,y)\,p(dy)}{t}\right)$ is lower semi continuous on $\mathcal{P}(X)$. The function $F$ is also bounded from below by $m=\inf_X f$. Moreover its sub-level sets are compact. Indeed, for all $r\geq m$, it holds
\[
\{F\leq r\} \subset \left\{ p \in \mathcal{P}(X) : \int d(x,y)\,p(dy) \leq C_{t,r}\right\},\quad \text{with}\quad C_{t,r}=t\alpha^{-1}\left(\frac{r-m}{t}\right).
\]
In particular, if $p\in \{F\leq r\}$, then $p(B(x,R)^c) \leq C_{t,r}R^{-1}$, for all $R>0$. Since balls in $X$ are assumed to be compact, the compactness of $\{F\leq r\}$ follows from Prokhorov theorem. Since $F$ is lower semicontinuous, bounded from below and has compact sub-level sets, $F$ attains its minimum and so $m_f(t,x)$ is not empty.
\end{proof}

\begin{proof}[Proof of Lemma \ref{Lem:convhull}]
Fix $f \colon X \to \RR$ bounded from below and lower semicontinuous, $x \in X$ and $u \in \RR^+$.
According to \textit{e.g.} \cite{HUL}[Proposition B.2.5.1],
\[
\widetilde{f}_x(u)=\inf \left\{\int_{I_x} f_x(w)\,q(dw) : q\in \mathcal{P}_u(I_x)\text{ with finite support}\right\}.
\]
Applying Caratheodory's Theorem (see \textit{e.g.} \cite{HUL}[Theorem A.1.3.6]), ones sees that one can assume that the infimum is over probability measures $q$ charging at most three points.
Let us explain how to reduce to two points.

Fix $\epsilon>0$ ; there exist $w_1,w_2,w_3 \in I_x$, and
$\lambda_1,\lambda_2,\lambda_3 \in [0,1]$ with $\sum_i \lambda_i=1$
such that $u=\lambda_1 w_1 + \lambda_2w_2 +\lambda_3 w_3$ and
$$\widetilde{f}_x(u)\geq \lambda_1f_x(w_1)+\lambda_2f_x(w_2)+\lambda_3f_x(w_3) -\epsilon$$
Without loss of generality we can assume that $w_1<w_2<w_3$, and for example that $w_1\leq u\leq w_2$ (the other case is similar).
Then there exist $a,b \in[0,1]$ such that
$u=aw_1+(1-a)w_2=bw_1+(1-b)w_3$. Then it is not difficult to check that there is a unique $\lambda \in [0,1]$ such that
$\lambda_1 = \lambda a +(1-\lambda)b$, $\lambda_2 = \lambda (1-a)$ and $\lambda_3 = (1-\lambda)(1-b)$.
Therefore it holds
$u=(\lambda a+(1-\lambda) b) w_1+\lambda(1-a)w_2+(1-\lambda)(1-b)w_3$
and
$$
\widetilde {f}_x(u)\geq (\lambda a+(1-\lambda) b) f_x(w_1)+\lambda(1-a)f_x(w_2)+(1-\lambda)(1-b)f_x(w_3)-\epsilon.
$$
By definition of $\widetilde{f}_x(u)$, necessarily,
$$
\widetilde {f}_x(u)\leq\min_{s \in [0,1]} \{(s a+(1-s) b) f_x(w_1)+s(1-a)f_x(w_2)+(1-s)(1-b)f_x(w_3) \}.
$$
Since, in the right hand side of the latter, the function of $s$ that needs to be minimized is an affine function, the minimum is
reached at $s=0$ or $s = 1$.
Therefore
$$
\widetilde {f}_x(u)\geq \min\{af_x(w_1)+(1-a)f_x(w_2), bf_x(w_1)+(1-b)f_x(w_3)\}-\epsilon
$$
which proves that, for all $\epsilon>0$, there exists $q \in \mathcal{P}_2(I_x)$ such that $\int v\,q(dv)=u$ and
$\int f_x(v)\,q(dv) \geq \widetilde{f}_x(u)\geq \int_{I_x} f_x(v)\,q(dv)-\epsilon$. Since $\epsilon>0$, this completes the proof.
\medskip

Now let us prove that $\widetilde{f}_x$ is continuous on $\widetilde{I}_x.$
By definition, $\widetilde{f}_x$ is a convex function on the closed interval $\widetilde{I}_x$, thus it is continuous on the interior of $\widetilde{I}_x$.
Hence it only  remains to prove that $\widetilde{f}_x$ is continuous at $0$ and, in case $I_x$ is bounded, at $b=\max I_x.$
We only give the proof of the continuity at $0$, the other case is similar.

Take $x_o \in X\setminus\{x\}$ and let $u_o=d(x,x_o)\in I_x\setminus\{0\}$. Since $\widetilde{f}_x$ is convex, on $\widetilde{I}_x$, it holds, for all $0\leq u \leq u_o$
\[
\widetilde{f}_x(u) = \widetilde{f}_x \left( \frac{u}{u_o}.u_o + (1-\frac{u}{u_o}).0 \right) \leq \frac{u}{u_o} \widetilde{f}_x (u_o) + \left(1-\frac{u}{u_o}\right) \widetilde{f}_x(0).
\]
Thus letting $u \to 0^+$, one gets that $\limsup_{u\to0^+} \widetilde{f}_x (u) \leq \widetilde{f}_x(0).$
Now, we prove that $\liminf_{u\rightarrow 0^+} \widetilde{f}_x(u) \geq \widetilde{f}_x(0)$.
 Thanks to the lower semicontinuity of $f$, for all $\epsilon \in (0,1)$, there exists $\eta$, for all $y\in B(x,\eta)$, $f(y)\geq f(x)-\epsilon$. Thus, from the definition of $f_x$, it follows that for all $u\in [0,\eta)$,
$$
f_x(u)\geq f_x(0)-\epsilon.
$$
On the other hand, if $m$ is a lower bound for $f$, then $f_x(u) \geq m$ for all $u\in [0,\infty)$.
Therefore, it holds
\[
f_x(u) \geq (f_x(0)-\epsilon)\mathbf{1}_{[0,\eta)}(u) + m \mathbf{1}_{[\eta,\infty)} := g_\epsilon(u),\qquad \forall u\in [0,\infty),
\]
(here we use that by definition $f_x(u) = +\infty$ when $u\notin I_x$).
Taking a smaller $m$ if necessary, one can assume that $f_x(0)-\epsilon > m$ for all $\epsilon \in (0,1)$. Now consider, the affine function $h_\epsilon$ joining $(0,f_x(0)-\epsilon)$ to $(\eta, m)$. It is clear that $g_\epsilon \geq h_\epsilon$ on $[0,\infty).$ Therefore, by definition of $\widetilde{f}_x$ as the greatest convex function below $f_x$, it holds $\widetilde{f}_x \geq h_\epsilon$ on $[0,\infty).$ In particular,
\[
\liminf_{u \to 0^+} \widetilde{f}_x(u) \geq \liminf_{u \to 0^+} h_\epsilon(u) = f_x(0)-\epsilon.
\]
Since $\epsilon$ is arbitrary, one concludes that $\liminf_{u \to 0^+} \widetilde{f}_x(u) \geq f_x(0) \geq \widetilde{f}_x(0).$ In conclusion, $\lim_{u \to 0^+} \widetilde{f}_x(u) = \widetilde{f}_x(0)= f_x(0),$ which completes the proof.
\end{proof}

%
%
%

\subsection{Properties of the gradient $\widetilde \nabla$}

In this section we collect some useful facts on the gradient $\widetilde \nabla$. Our first result is some sort of chain rule formula for $\widetilde \nabla$.

\begin{Pro}\label{propnabla}
Let $f \colon X \mapsto \RR$ and $G \colon f(X) \mapsto \RR$.
\begin{description}
\item $(i)$ If $G$ is non-decreasing then
$|\widetilde \nabla G\circ f|(x)\leq |\widetilde \nabla f|(x)|\widetilde \nabla G|\left(f(x)\right)$, $x\in X$.
\item$(ii)$ If $G$ is non-increasing then
$|\widetilde \nabla G\circ f|(x)\leq |\widetilde \nabla (-f)|(x)|\widetilde \nabla G|\left(f(x)\right)$,
$x\in X$.
\end{description}
Here, $|\widetilde \nabla G|(u):=\sup_{v \in \RR} \frac{[G(v)-G(u)]_-}{|v-u|}$, $u \in \RR$, with $|\cdot|$ being the absolute value.
\end{Pro}

\begin{proof}
Fix $x \in X$ and assume that $G$ is non-decreasing. Let $y\in X$ be such that $f(x) > f(y)$ (if $\{y \in X : f(x) > f(y)\} = \emptyset$ then $|\widetilde \nabla G\circ f|(x) = |\widetilde \nabla f|(x) = 0$ and there is nothing to prove). Since $G$ is non-decreasing $G(f(x))\geq G( f(y))$ so that
\begin{align*}
\frac{G\left(f(x)\right)-G\left(f(y)\right)}{d(x,y)}
\leq
 \frac{f(x)-f(y)}{d(x,y)}\frac{G\left(f(x)\right)-G\left(f(y)\right)}{f(x)-f(y)}
\leq
 |\widetilde \nabla f|(x)|\widetilde \nabla G|\left(f(x)\right).
\end{align*}
Taking the supremum over all $y$ such that $f(x)> f(y)$ leads to the desired conclusion of Item $(i)$.

The proof of Item $(ii)$ is similar. Let $y \in X$ be such that $f(y)>f(x)$, then
$G( f(y))\leq G(f(x))$ (since $G$ is non-increasing) so that
\begin{align*}
\frac{G\left(f(x)\right)-G\left(f(y)\right)}{d(x,y)}
&=
 \frac{(-f)(x)-(-f)(y)}{d(x,y)}\frac{G\left(f(x)\right)-G\left(f(y)\right)}{|f(y)-f(x)|}\\
&\leq
 |\widetilde \nabla (-f)|(x)|\widetilde \nabla G|\left(f(x)\right) .
\end{align*}
The result follows by taking the supremum over all $y \in X$ such that $f(y)>f(x)$.
\end{proof}

\begin{Rem} \label{remark}
Observe that $|\widetilde \nabla (C f)|(x)=C|\widetilde \nabla f|(x)$ for $C>0$, while
$|\widetilde \nabla (C f)|(x)=-C|\widetilde \nabla (-f)|(x)$ for $C<0$. Because of the negative part entering in its definition, in general $|\widetilde \nabla (-f)| \neq |\widetilde \nabla f|$.
\end{Rem}

The next proposition gives some results on the action of the gradient $\widetilde \nabla$ onto the operator $\widetilde Q_t$ and relates the gradient of $f$ to the usual derivative of $\widetilde f$.

\begin{Pro}\label{eq6}
Let $f$ be a lower semi-continuous function bounded from below.
\begin{description}
\item{(i)}
For all $x\in X$, all $t>0$ and all
$p\in m_f(t,x)$, it holds
\begin{equation}
|\widetilde{\nabla}\tQ_tf|(x)\leq \alpha' \left(\frac{\int d(x,y)\,p(dy)}{t} \right).
\end{equation}
\item{(ii)}
Assume that $f$ reaches its minimum at a unique point $x_o \in X$, then for all $x\in X\setminus\{x_o\}$, it holds
\begin{equation}\label{fxu2pro}
|\widetilde{\nabla}f|(x)=|\widetilde{f}_x^{\,\, '}(0)|,
\end{equation}
and $|\widetilde{\nabla}f|(x_o)=0$. Moreover, if $f$ reaches its minimum in two or more points, or if $f$ does not reach its minimum, then \eqref{fxu2pro} holds for all $x \in X$.
\end{description}
\end{Pro}

\begin{Rem}
Observe that, if $f$ reaches its minimum at a unique point $x_o$, then it could be that
$\widetilde{f}_{x_o}^{\,\,'}(0) \neq 0$. For example consider, on $X = \RR^+$, $f(x)=x$ that reaches its minimum at $x_o=0$.
Trivially $\widetilde{f}_{x_0}(x)=x$ for all $x \in X$ so that  $\widetilde{f}_{x_o}^{\,\,'}(0)=1$. Hence, there is no hope for \eqref{fxu2pro}
to be true at $x_o$ in general.
\end{Rem}

\begin{proof}
First let us prove item $(i)$.
Consider $y$ such that $\tQ_tf(y)< \tQ_tf(x)$ (if there is no such $y$, then $|\widetilde \nabla \tQ_tf|(x)=0$ and there is nothing to prove). By Lemma \ref{mfnonempty}, there exist $p_o\in m_f(t,x), p_1\in m_f(t,y)$ and according to Item $(ii)$ of Theorem \ref{the1}, $u_o = \int d(x,z)\,p_0(dz) \in \widetilde {m}_f(t,x)$ and $u_1=\int d(y,z)\,p_1(dz)\in \widetilde {m}_f(t,y)$ and it holds
\begin{equation} \label{ihp1}
\tQ_tf(x)=\int f\,dp_o+t\alpha(u_o/t) \qquad \mbox{and} \qquad
\tQ_tf(y)=\int f\,dp_1+t\alpha(u_1/t) .
\end{equation}
Now, set
$p_\lambda:=(1-\lambda)p_o+\lambda p_1$, $\lambda \in [0,1]$,
$u := \int d(x,z)\,p_1(dz)$ and observe that, by definition of $\tQ_t$,
$$
\tQ_tf(x)\leq
  \int f\,dp_\lambda+t\alpha\left(\frac{\int d(x,z)\,p_\lambda(dz)}{t} \right) =
\int fdp_\lambda+t\alpha \left(\frac{\lambda u+(1-\lambda)u_o}{t} \right).
$$
Since the latter holds for all $\lambda \in [0,1]$ the function
$$
g \colon \lambda \mapsto \int f \,dp_\lambda+t\alpha\left(\frac{\lambda u+(1-\lambda)u_o}{t}\right)-\tQ_tf(x)
$$
is always non-negative. Therefore, since $g(0)=0$, $g'(0) =(\int f\,dp_1-\int f\,dp_o)+(u-u_o)\alpha'(u_o/t) \geq 0$ which ensures that
\begin{equation}\label{eqg0}
\int f\,dp_o-\int f\,dp_1 \leq  (u-u_o)\alpha'(u_o/t) .
\end{equation}
On the other hand, since $d(x,z)\leq d(x,y)+d(y,z)$, it holds
$u = \int d(x,z)\,p_1(dz) \leq \int(d(x,y)+d(y,z))\,p_1(dz)\leq u_1+d(x,y)$. As a consequence,
it holds
\begin{equation} \label{ihp2}
u- u_1 \leq d(x,y) .
\end{equation}
Thanks to \eqref{ihp1}, \eqref{eqg0} and \eqref{ihp2} together with the fact that $\alpha' \geq 0$,
for all $y$ such that $\tQ_tf(x)>\tQ_tf(y)$, it holds
\begin{align*}
[\tQ_tf(y)-\tQ_tf(x)]_-
&=
\tQ_tf(x)-\tQ_tf(y)\\
& =
\int f\,dp_o-\int f\,dp_1 + t \left(\alpha\left(\frac{u_o}{t}\right)-\alpha\left(\frac{u_1}{t}\right)\right)\\
& \leq
(u-u_o)\alpha'(\frac{u_o}{t})+t\left(\alpha\left(\frac{u_o}{t}\right)-\alpha\left(\frac{u_1}{t}\right)\right)\\
& \leq d(x,y)\alpha' \left(\frac{u_o}{t} \right)
+(u_1-u_o)\alpha' \left(\frac{u_o}{t} \right)+t \left(\alpha \left(\frac{u_o}{t} \right)-\alpha\left(\frac{u_1}{t} \right) \right).
\end{align*}
Therefore, by convexity of $\alpha$, we conclude that
$(u_1-u_o)\alpha'(\frac{u_o}{t})+t(\alpha(\frac{u_o}{t})-\alpha(\frac{u_1}{t}))\geq 0$ and in turn that for all $x, y\in X$,
$[\tQ_tf(y)-\tQ_tf(x)]_- \leq d(x,y)\alpha'(\frac{u_o}{t})$
which leads to the expected result by taking the supremum over $y \neq x$.
\medskip

Now we turn to the proof of Item $(ii)$. Fix $x \in X$. The proof relies on the existence of a point $y \neq x$ such that $f(y) \leq f(x)$.
Such an existence is guaranteed for all $x \in X$ (resp. for all $x \in X \setminus \{x_o\}$) when $f$ does not reach its minimum or reaches its minimum in more than two points (resp. when $f$ reaches its minimum at a unique point $x_o$).
Given such a point $y$, by definition of $\widetilde{f}_x$, we have $\widetilde{f}_x(0)=f(x) \geq f(y)\geq \widetilde{f}_x\left(d(x,y)\right)$. Thanks to the convexity of $\widetilde{f}_x$, the slope function $u \mapsto \frac{\widetilde{f}_x(u)-\widetilde{f}_x(0)}{u}$ is non-decreasing. It follows that
$$
\widetilde{f}_x^{\,\,'}(0)=\lim_{u\rightarrow 0^+}\frac{\widetilde{f}_x(u)-\widetilde{f}_x(0)}{u}=\inf_{u>0}\frac{\widetilde{f}_x(u)-\widetilde{f}_x(0)}{u}\leq \frac{\widetilde{f}_x\left(d(x,y)\right)-\widetilde{f}_x(0)}{d(x,y)} \leq 0 .
$$
Taking the absolute value, we get
$$
|\widetilde{f}_x^{\,\,'}(0)|=\sup_{u>0}\frac{\widetilde{f}_x(0)-\widetilde{f}_x(u)}{u}.
$$
Observe that, according to Lemma \ref{Lem:convhull}, for all $u>0$, $\widetilde{f}_x(u)=\inf \int  f\,dp$ where the infimum is running over all $p\in \mathcal{P}_2(X)$ such that $\int d(x,\,\cdot\,)\,dp=u$.
Hence, setting $p=\lambda \delta_{y_1} + (1-\lambda) \delta_{y_2}$, $y_1,y_2 \in X$, $\lambda \in [0,1]$, we have (recall that $\widetilde{f}_x(0)=f(x)$)
\begin{align*}
|\widetilde{f}_x^{\,\,'}(0)|
&=
\sup_{u>0}\frac{\widetilde{f}_x(0)-\widetilde{f}_x(u)}{u}.\\
&=
\sup_{u>0}\sup_{\genfrac{}{}{0pt}{}{y_1,y_2 \in X, \lambda\in[0,1]s.t}{\ \lambda d(x,y_1)+(1-\lambda)d(x,y_2)=u}}\frac{f(x)-\left(\lambda f(y_1)+(1-\lambda)f(y_2)\right)}{u}\\
&=
\sup_{y_1, y_2\in X,\lambda\in[0,1]}\frac{\lambda (f(x)- f(y_1))+(1-\lambda)(f(x)-f(y_2))}{\lambda d(x,y_1)+(1-\lambda)d(x,y_2)}\\
&=
\sup_{y\neq x}\frac{f(x)-f(y)}{d(x,y)}=|\widetilde{\nabla}f|(x),
\end{align*}
where the last equality comes from the fact that the function $\lambda \mapsto \frac{\lambda a + (1-\lambda) b}{\lambda c +(1-\lambda )d}$ (with $c,d>0$ and $a,b\in \RR$) is monotone on $[0,1].$ This proves \eqref{fxu2pro}. That $|\widetilde \nabla f|(x_o)=0$ is a direct consequence of the definition of the gradient.
\end{proof}
\subsection{Obstruction to the semi-group property of the usual inf-convolution operator $Q_t$, on graphs}

In this section we prove that, on a graph and under very mild assumptions,
there is no hope of finding a family of mappings $(D_t)_{t > 0}$ such that  $Q_tf(x):=\inf_{y\in V}\{f(y)+D_t(y,x)\}$ satisfies the usual semi-group property $Q_{t+s}=Q_t( Q_s)$.

More precisely, we have the following result.

\begin{Pro}
Let $G=(V,E)$ be a finite graph.
Assume we are given a family of mappings $D_t \colon V \times V \to \mathbb{R}^+$, $t>0$ that satisfies
$D_t(x,x)=0$ for all $x \in V$ and all $t>0$. Assume furthermore that for any $f \colon V \to \RR$ and any $x \in V$,
$Q_tf(x):=\inf_{y\in V}\{f(y)+D_t(y,x)\} \rightarrow f(x)$ when $t\rightarrow 0$.
Then, there exists $f$, $x \in V$ and $t,s >0$ such that $Q_{t+s}f(x) \neq Q_t( Q_s f)(x)$.
\end{Pro}

\begin{proof}
By contradiction assume that for all $f$ bounded on $V$, all $x \in X$ and $s,t>0$, it holds
$Q_tQ_sf=Q_{t+s}f$. The proof is based on the following claims.

\begin{Claim} \label{claim2}
For all $x,z \in V$, all $s<r \in(0,\infty)$, it holds
$D_{r}(z,x)=\min_{y \in V} \left\{ D_s(z,y)+D_{r-s}(y,x) \right\}$.
\end{Claim}

\begin{Claim} \label{claim1}
For all $x,z \in V$, the map $(0,\infty) \ni t \mapsto D_t(z,x)$ is non-increasing and, if $x \neq z$,
$D_t(z,x) \to \infty$ as $t$ goes to $0$.
\end{Claim}

We postpone the proof of the above claims to end the prove of the proposition.

Fix $x,z \in V$, $x \neq z$. Then,  by Claim \ref{claim2}, for all $s \in (0,1)$, it holds
\begin{align*}
D_1(z,x)
& =
\min_{y \in V} \{ D_s(z,y) + D_{1-s}(y,x) \} \\
& =
\min\left( D_{1-s}(z,x) ; \min_{y \neq z} \{D_s(z,y) + D_{1-s}(y,x)\} \right) .
\end{align*}
By Claim \ref{claim1} and since the graph is finite, $\lim_{s \to 0} \min_{y \neq z}\{ D_s(z,y) + D_{1-s}(y,x)\} = \infty$. Hence, there exists $s_o \in (0,1)$ such that, for $s < s_o$, $D_1(z,x)=D_{1-s}(z,x)$ so that $u_o:= \sup\{ u \in (0,1) : D_{1-u}(z,x)=D_1(z,x)\}$ is well-defined thanks to Claim \ref{claim1}. By a similar argument, there exists $s_1 \in (0,1-u_o)$ such that $D_{1-u_o-s}(z,x)=D_{1-u_o}(z,x)$ for all $s < s_1$. This contradicts the definition of $u_o$ and ends the proof of the proposition provided that we prove Claim \ref{claim1} and Claim \ref{claim2}.

\begin{proof}[Proof of Claim \ref{claim2}]
Since $D_t(x,z)$ is non-negative and $D_t(x,x)=0$, the claim is trivial if $x=z$. Assume that $x\neq z$. Let $s < r$ and consider $f \colon V \to \RR$ defined by
$f(z)=0$ and $f(y) = D_r(z,x) + 1$ for all $y \neq z$. Then
\begin{align*}
Q_r f(x)
& =
\min_{y \in V} \{ f(y) + D_r(y,x) \} = \min\left(D_r(z,x) ; \min_{y \neq z} \{ f(y) + D_r(y,x) \} \right) \\
& =  D_r(z,x) .
\end{align*}
On the other hand, by the semi-group property, similarly (necessarily $u=z$) it holds
\begin{align*}
Q_r f(x)
& =
Q_{r-s}(Q_s f)(x) = \min_{u,y \in V} \{ f(u) + D_s(u,y) + D_{r-s}(y,x) \} \\
& =
\min_{y \in V} \{ D_s(z,y) + D_{r-s}(y,x) \}
\end{align*}
which leads to the thesis.
\end{proof}

\begin{proof}[Proof of Claim \ref{claim1}]
If $x=z$, the map $t \mapsto D_t(z,x)$ is constant and so there is nothing to prove. Assume that $x \neq z$.
By Claim \ref{claim2} we have for $s<r$ (take $y=x$),
$D_r(z,x) = \inf_{y \in V} \{ D_s(z,y) + D_{r-s}(y,x) \}\leq D_s(z,x)$
which proves that $t \mapsto D_t(z,x)$ is non-increasing and that the limit $\lim_{r \to 0} D_r(z,x)$ exists in $[0,\infty]$.
For $M>0$, let $f \colon V \to \RR$ be defined by $f(z)=0$, $f(x)=M$ and $f(y)=M+1$ for all $y\neq z,x$.
Then
\begin{align*}
Q_r f(x)
& =
\min_{y \in V} \{ f(y) + D_r(y,x) \} = \min\left(D_r(z,x) ; f(x) ; \min_{y \neq z,x} \{ f(y) + D_r(y,x) \} \right) \\
& =  \min\left(D_r(z,x) ; M \right) \leq \frac{1}{2} \left( D_r(z,x) +  M\right) .
\end{align*}
Now, by assumption $Q_r f(x) \to f(x)=M$ as $r$ goes to 0 so that, taking the limit in the latter guarantees that
$\lim_{r \to 0} D_r(z,x) \geq M$ which ends the proof of Claim \ref{claim1} since $M$ is arbitrarily large.
\end{proof}
The proof of the proposition is complete.
\end{proof}

\section{Hamilton-Jacobi equation: Proof of Theorem \ref{thm:main}} \label{sec:proof}

This section is dedicated to the proof of Theorem \ref{thm:main}. Actually we shall prove a more general result involving a general choice of the function $\alpha$, not only $\alpha(x)=\frac{1}{2} x^2$ as stated in Theorem \ref{thm:main}. More precisely, we shall prove the following (recall that
$\alpha^*$ is the Fenchel-Legendre transform of $\alpha$ defined in Section \ref{preliminaries}).

\begin{The}\label{thm:main2}
Let $f \colon X \rightarrow \RR$ be a lower semi-continuous function bounded from below.
Then, for all $x \in X$, it holds
\begin{description}
\item $(i)$ For all $t>0$, $\frac{\partial}{\partial t}\tQ_tf(x)+ \alpha^*\left( |\widetilde{\nabla}\tQ_tf|(x) \right) \leq 0$.
\item $(ii)$ Assume that $\alpha^*$ is well define on $[0,l)$, (i.e $\forall x\in [0,l)$, $\alpha^*(l)<\infty$.) Then for all $x$ such that $|\widetilde{\nabla} f|(x)\in [0,l)$, $\lim_{t \to 0} \tQ_t f = f$ and
it holds
$$
\frac{\partial}{\partial t}\tQ_tf(x)|_{t=0}+ \alpha^* \left( |\widetilde{\nabla}f|(x) \right) = 0.
$$
\end{description}
\end{The}

\begin{Rem}
\begin{description}
\item In Item $(ii)$, if $\lim_{x\rightarrow \infty}\alpha(x)/x=\infty$, we can take $l=\infty$, then the latter equation holds for almost every $x\in X$.
\item If $f$ is $l-\epsilon$-lipschiz then $|\widetilde{\nabla}f|(x)<l$ and the latter equality holds. Moreover, if there exists $h$ such that $\alpha'(h)=l$, then the latter holds for all $x$ such that $|\widetilde{\nabla} f|(x)\in [0,l]$.
\end{description}
\end{Rem}

\begin{proof}
We will first prove Item $(i)$.
On the one hand, by Theorem \ref{the1}, for all $t>0$, it holds
$$
\frac{\partial}{\partial t} \tQ_tf(x)= - \beta\left(\frac{u_o}{t}\right), \qquad x \in X
$$
where $u_o \in \widetilde{m}_f(t,x)$. On the other hand,  since $\alpha^*$ is non-decreasing, Proposition \ref{eq6} ensures that
$$
\alpha^* \left(|\widetilde{\nabla}\tQ_tf|(x) \right)
\leq
\alpha^*\left(\alpha'\left(\frac{u_o}{t}\right)\right) .
$$
In order to conclude, it is enough to observe that, the function $G:=y\mapsto y\alpha'\left(\frac{u_o}{t}\right)-\alpha(y)$ is a concave function and $G'\left(\frac{u_o}{t}\right)=0$. Hence,
$$
\alpha^*\left(\alpha'\left(\frac{u_o}{t}\right) \right)= \sup_{y\in \RR}\left\{y\alpha'\left(\frac{u_o}{t}\right)-\alpha(y)\right\}=\frac{u_o}{t}\alpha'\left(\frac{u_o}{t}\right)-\alpha\left(\frac{u_o}{t}\right)=\beta\left(\frac{u_o}{t}\right).
$$

Now we turn to the proof of Item $(ii)$.
If $x=x_o$ is a minimum of $f$ (if any), then (observe that $\tQ_t f(x_o) = f(x_o)$ for all $t >0$) it is easy to see that $\frac{\partial}{\partial t}\tQ_tf(x)|_{t=0}:= \lim_{t \to 0}\frac{\tQ_t f(x)-f(x)}{t}= \alpha^* \left( |\widetilde{\nabla}f|(x) \right)=0$ and the claim follows. For the remaining of the proof we assume that $x\in X$ is not a minimum of $f$.
Thanks to Theorem~\ref{the1}, for all $t>0$, it holds
$$
\frac{\tQ_tf(x)-f(x)}{t}=\frac{Q_t\widetilde{f}_x(0)-\widetilde{f}_x(0)}{t}=\frac{\widetilde{f}_x(u)-\widetilde{f}_x(0)}{t}+\alpha\left(\frac{u}{t}\right),
$$
where $u\in \widetilde{m}_f(t,x)$.

Let us prove that $u>0$. Since $x$ is not a minimum of $f$,  there exists $y\in X$ such that $f(y)<f(x)$.
Fix $t>0$, by the very definition of $\widetilde {Q}_t$, for all $\lambda\in [0,1]$, it holds that $\tQ_tf(x)\leq (1-\lambda) f(x)+\lambda f(y)+t\alpha\left(\frac{\lambda d(x,y)}{t}\right)$ (choose $p=(1-\lambda) \delta_x + \lambda \delta_y$).
Define $G: [0,1] \ni \lambda \mapsto (1-\lambda) f(x)+\lambda f(y)+t\alpha\left(\frac{\lambda d(x,y)}{t}\right)$. Then
$G'(0)=\alpha'(0)d(x,y)+f(y)-f(x)=f(y)-f(x)<0$. Thus, there exist $\lambda\in (0,1)$ such that $\tQ_tf(x)\leq G(\lambda)<G(0)=f(x)$. Hence  $\widetilde{f}_x(u) \leq \tQ_tf(x)<f(x)=\widetilde{f}_x(0)$ and therefore $u>0$.

According to Lemma \ref{Lem:convhull}, for all $x\in X$, $\widetilde{f}_x$ is convex and continuous on $\widetilde{I}_x$. It follows that
$\frac{\widetilde{f}_x(u)-\widetilde{f}_x(0)}{u}\geq \widetilde{f}_x^{\,\,'}(0)$. Since $\widetilde{f}_x(u)\leq Q_t\widetilde{f}_x(0)\leq \widetilde{f}_x(0)$, we have that
$\frac{\widetilde{f}_x(u)-\widetilde{f}_x(0)}{u}$ is non-positive and $\frac{\widetilde{f}_x(0)-\widetilde{f}_x(u)}{u}\leq |\widetilde{f}_x^{\,\, '}(0)|$.
Hence,
\begin{align*}
\frac{f(x)-\tQ_tf(x)}{t}
&=
\frac{\widetilde{f}_x(0)-\widetilde{f}_x(u)}{t}-\alpha\left(\frac{u}{t}\right)=\frac{\widetilde{f}_x(0)-\widetilde{f}_x(u)}{u}\frac{u}{t}-\alpha\left(\frac{u}{t}\right)\\
&\leq \alpha^*\left(\frac{\widetilde{f}_x(0)-\widetilde{f}_x(u)}{u}\right)\leq \alpha^*\left(|\widetilde{f}_x^{\,\,'}(0)|\right)
\end{align*}
where the last inequality comes from the fact that $\alpha^*$ is non-decreasing.
This leads to
\begin{equation}\label{eq1022}
\liminf_{t\rightarrow 0}\frac{\tQ_tf(x)-f(x)}{t}\geq -\alpha^*\left(|\widetilde{f}_x^{\,\,'}(0)|\right),
\end{equation}
by passing to the limit.

Next, we prove that $\limsup_{t\rightarrow 0}\frac{\tQ_tf(x)-f(x)}{t}\leq-\alpha^* \left( |\widetilde{f}_x^{\,\,'}(0)| \right)$.
 By convexity of $\widetilde{f}_x$, for all $h\in(0,u)$, it holds
\begin{equation}\label{eqconvex}
\frac{\widetilde{f}_x(u)-\widetilde{f}_x(0)}{u}\leq \frac{\widetilde{f}_x(u)-\widetilde{f}_x(u-h)}{h}.
\end{equation}
On the other hand, since (by definition of $u$) $\widetilde{f}_x(u)+t\alpha\left(\frac{u}{t}\right)\leq \widetilde{f}_x(u-h)+t\alpha\left(\frac{u-h}{t}\right)$, we have
\begin{equation}\label{eqconvex2}
 \frac{\widetilde{f}_x(u)-\widetilde{f}_x(u-h)}{h}\leq \frac{t\left(\alpha\left(\frac{u-h}{t}\right)-\alpha\left(\frac{u}{t}\right)\right)}{h}.
\end{equation}

According to \eqref{eqconvex} and \eqref{eqconvex2}, for all $h\in(0,u)$, it holds:
\begin{align*}
\frac{\tQ_tf(x)-f(x)}{t}
 =
\frac{\widetilde{f}_x(u)-\widetilde{f}_x(0)}{t}+\alpha\left(\frac{u}{t}\right)
\leq
\frac{u}{t}\frac{\alpha\left(\frac{u-h}{t}\right)-\alpha\left(\frac{u}{t}\right)}{h/t}+\alpha\left(\frac{u}{t}\right).
\end{align*}
Let $h$ goes to 0, we get that
\begin{equation} \label{faim}
\frac{\tQ_tf(x)-f(x)}{t}\leq -\frac{u}{t}\alpha'\left(\frac{u}{t}\right)+\alpha\left(\frac{u}{t}\right)=-\beta\left(\frac{u}{t}\right)=-\alpha^*\left(\alpha'\left(\frac{u}{t}\right)\right)
\end{equation}
where we recall that $\beta$ is defined in Section \ref{sec:qt}.
Hence, it is enough to prove that $\lim_{t\rightarrow 0}\alpha'\left(\frac{u}{t}\right)=|\widetilde{f}_x^{\,\,'}(0)|$.
Since $\widetilde{f}_x$ is convex, it is right and left differentiable at every point. Hence taking the left derivative of
$v \mapsto \widetilde{f}_x(v)+t\alpha\left(\frac{v}{t}\right)$,
for all $t \in\RR^+$ and all $u \in \widetilde{m}_f(t,x)$, we have
$$
\alpha'\left(\frac{u}{t}\right) \leq -\frac{d}{du_-}\widetilde{f}_x(u).
$$

Let $l:=\lim_{x\rightarrow \infty}\alpha'(x)$, it is easy to see that $\alpha^*(x)<\infty$ when $x\leq l$ and $=\infty$ when $x>l$.
By  Item $(ii)$ of Proposition \ref{eq6} and convexity  of $\widetilde{f}_x$ and Equation \eqref{fxu2pro}, there exists $h_1<l$ such that the following holds:
  $$
  \alpha'\left(\frac{u}{t}\right)
  \leq
  - \frac{d}{du_-}\widetilde{f}_x^{\,\,'}(u)
  \leq -\widetilde{f}_x^{\,\,'}(0)=|\widetilde{\nabla}f|(x) \leq \alpha'(h_1) .
  $$
By convexity of $\alpha$, the latter inequality leads to $\frac{u}{t}\leq h_1$ for all $t>0$.
We conclude from the above argument that $u \in \widetilde{m}(t,x)$ goes to $0$ as $t$ goes to $0$.

Now, taking the right  derivative of
$v \mapsto \widetilde{f}_x(v)+t\alpha\left(\frac{v}{t}\right)$,
for all $t \in\RR^+$ and all $u \in \widetilde{m}_f(t,x)$, we have
$$
\alpha'\left(\frac{u}{t}\right) \geq -\frac{d}{du_+}\widetilde{f}_x(u).
$$
Since $\lim_{u\rightarrow 0} \frac{d}{du_+}\widetilde{f}'_x(u) = \widetilde{f}_x^{\,\,'}(0)$ and using the monotonicity and the (right) continuity of $\alpha^*$
when $t$ goes to $0$, we have thanks to \ref{faim}
\begin{equation}\label{eq10221}
\limsup_{t\rightarrow 0} \frac{\tQ_tf(x)-f(x)}{t} \leq -\alpha^*\left(|\widetilde{f}_x^{\,\,'}(0)|\right)
\end{equation}
This combined with \ref{eq1022} and Proposition\ref{eq6} leads to the desired result.







\end{proof}

\section{Functional inequalities}\label{functional inequality}

In this section we shall first introduce different functional inequalities (of Poincar\'e and log-Sobolev type related to the gradient $\widetilde \nabla$) and two transport-entropy inequalities. Then, following \cite{BGL} on the one hand, and \cite{BL} on the other hand, by means of our main result on the Hamilton-Jacobi equation (Theorem \ref{thm:main2}) we shall prove some relations between such inequalities. For simplicity and to avoid unnecessary technical assumptions and proofs, we shall mainly deal with the quadratic or quadratic-linear costs.
However, most of the results below can be extended to more general situations.
\medskip

We start with some definitions. One says that $\mu \in \mathcal{P}(X)$ satisfies the Poincar\'e inequality, respectively the modified log-Sobolev inequality\footnote{We observe that the terminology here is not optimal since there already exist, in the literature, many different inequalities called modified log-Sobolev inequality that have \textit{a priori} no relation between them, and no relation with our definition.} of type I and type II, respectively the weak transport-entropy inequality of type I and type II, if there exists a constant $C \in (0, \infty)$ such that for all $f \colon X \to \RR$ bounded it holds
\begin{equation} \label{poincare}
\var_\mu(f) \leq C \int  |\widetilde \nabla f|^2  \,d\mu \qquad (\mbox{Poincar\'e Inequality}),
\end{equation}
respectively
\begin{equation} \label{logsob}
\ent_\mu(e^f) \leq C \int  \alpha^*\left(|\widetilde \nabla f|\right) e^f \,d\mu
\quad (\mbox{Modified log-Sob Ineq. of type I}),
\end{equation}
\begin{equation} \label{logsob2}
\ent_\mu(e^f) \leq C \int  \alpha^*\left(|\widetilde \nabla (-f)|\right) e^f\, d\mu
\quad (\mbox{Modified log-Sob Ineq. of type II}),
\end{equation}
respectively for all $\nu \in \mathcal{P}(X)$ it holds
\begin{equation} \label{t2}
\widetilde{T}_\alpha(\mu|\nu) \leq C H(\nu | \mu)
\qquad (\mbox{Weak transport-entropy Inequality of type I}),
\end{equation}
\begin{equation} \label{t22}
\widetilde{T}_\alpha(\nu|\mu) \leq C H(\nu | \mu)
\qquad (\mbox{Weak transport-entropy Inequality of type II}),
\end{equation}
where we recall that $\var_\mu(f):=\int  f^2 \,d\mu - \left( \int  f \,d\mu\right)^2$ is the variance of $f$ with respect to $\mu$, $\ent_\mu(e^f):=\int  f e^f \,d\mu - \int  e^f \,d\mu \log \int  e^f \,d\mu$ is the entropy of $e^f$ with respect to $\mu$, $H(\nu | \mu)= \ent_\mu(e^f)$ if $\nu \ll \mu$ and $e^f=d\nu/d\mu$, and $H(\nu|\mu)=\infty$ otherwise, while $\widetilde{T}_2(\mu|\nu)$ is defined in \eqref{weakt}. For general $\alpha$, we have
\begin{equation} \label{weaktbis}
\widetilde T_\alpha(\nu|\mu) := \inf \left\{ \int  \alpha \left( \int  d(x,y)\,p_x(dy)\right) \mu(dx) \right\},
\qquad \mu,\nu \in \mathcal{P}(X)
\end{equation}
 where the infimum is running over all couplings $\pi(dx,dy)=p_x(dy)\mu(dx)$ of $\mu,\nu$ (\textit{i.e.}\ $\pi$ is a probability measure on $X \times X$ with first marginal $\mu$ and second marginal $\nu$).
We stress that
$\widetilde{T}_2(\,\cdot\,|\,\cdot\,)$ is not symmetric so that \eqref{t2} is in general different from \eqref{t22}.
For further developments on transport-entropy inequalities involving
$\widetilde{T}_2$, we refer to \cite{Dualdiscrete}.

\subsection{Modified log-Sobolev inequality}\label{subsectionMLI}

In this section, we focus on the modified log-Sobolev inequalities \eqref{logsob}-\eqref{logsob2}.
As a first result we shall prove that, in the graph setting, some other (say classical) modified log-Sobolev inequality (which is known to be weaker than the usual log-Sobolev inequality \cite{BobkovTetali}, an inequality deeply studied in the literature) implies \eqref{logsob}.
Then, we may extend to our general setting the approach and some of the results of \cite{BGL} on the hypercontractivity of the Hamilton-Jacobi operator $\widetilde Q_t$. This will allow us to prove that, in particular, the modified log-Sobolev inequality \eqref{logsob} (resp. \eqref{logsob2}) implies the weak transport-entropy inequality \eqref{t2} (resp. \eqref{t22}).

%
%
%
%

\subsection{Connection with some classical inequalities, on graphs}

Given a (simple connected) graph $G=(V,E)$, recall that $K=(K(x,y))_{x,y \in V}$ is a matrix with positive entries if $K(x,y) \geq 0$ for all $x, y \in V$,  and that it is a Markovian matrix if in addition
$\sum_{y \in V} K(x,y)=1$ for all $x \in V$. Then, the couple
$(\mu,K)$ satisfies the (say) classical modified log-Sobolev inequality if there exists
a constant $C \in (0, \infty)$ such that for all $f \colon V \to \RR$ bounded it holds
\begin{equation} \label{usual-logsob}
\ent_\mu(e^f) \leq C \sum_{x,y \in V} (e^{f(y)} - e^{f(x)})(f(y)-f(x)) \mu(x) K(x,y) .
\end{equation}
The latter is known to be a consequence of Gross' Inequality that asserts that
\begin{equation} \label{gross}
\ent_\mu(f) \leq C' \sum_{x,y \in V} (f(y)-f(x))^2 \mu(x) K(x,y) \qquad \forall f \colon V \to \RR \quad\mbox{bounded} .
\end{equation}
More precisely Gross' Inequality \eqref{gross} with constant $C'$ implies the classical modified log-Sobolev inequality \eqref{usual-logsob} with constant $C \leq C'/4$, see \cite[Theorem 3.6]{BobkovTetali}.

\begin{Pro} \label{toto}
Let $\mu$ be a probability measure on a (simple connected) graph $G=(V,E)$ and $K$ be a matrix with positive entries.
Assume that there exists a constant $L$ such that $\sum_{y \in V} d^2(x,y) K(x,y) \leq L$ for all $x \in V$ and
that for all $x,y \in V$, $\mu(x)K(x,y)=\mu(y)K(y,x)$. Finally, assume that $(\mu,K)$ satisfies the classical  modified log-Sobolev inequality \eqref{usual-logsob} with constant $C$, respectively Gross' Inequality \eqref{gross} with constant $C'$. Then, $\mu$ satisfies the modified log-Sobolev inequality \eqref{logsob} with $\alpha(x)=\alpha^*(x)=x^2/2$ and constant $4LC$, respectively $LC'$.
\end{Pro}

\begin{Rem}
The condition $\mu(x)K(x,y)=\mu(y)K(y,x)$, $x,y \in V$, is known as the detailed balance condition in the physics literature
and means that the operator $K$, acting on functions, is symmetric in $\mathbb{L}^2(\mu)$. Most commonly one deals with a Markovian matrix with nearest neighbor jumps (meaning that $K(x,y)=0$ unless $d(x,y)=1$), which guarantees that $L=1$. In particular the hypotheses of the proposition are very commonly used and correspond to a lot of practical situations \cite{Saloff}.
\end{Rem}

\begin{proof}
The result involving the Gross' inequality is an immediate consequence of the result involving the classical modified log-Sobolev inequality since the former implies the latter with $C' \leq C/4$.

Hence, we only need to show that
$$
\sum_{x,y \in V} (e^{f(y)} - e^{f(x)})(f(y)-f(x)) \mu(x) K(x,y)
\leq
2L \sum_{x \in V} |\widetilde{\nabla}f|^2(x)e^{f(x)} \mu(x) .
$$
Since $(a-b)(e^a-e^b)\leq (a-b)^2 \max \{e^a,e^b\}$, we have
\begin{align*}
&\sum_{x,y \in V} (e^{f(y)} - e^{f(x)})(f(y)-f(x)) \mu(x) K(x,y) \\
& \leq
\!\!\!\sum_{\genfrac{}{}{0pt}{}{x,y \in V:}{f(x) \geq f(y)}} (f(y)-f(x))^2 e^{f(x)} \mu(x) K(x,y)
 + \!\!\! \sum_{\genfrac{}{}{0pt}{}{x,y \in V:}{f(y) \geq f(x)}} (f(x)-f(y))^2 e^{f(y)} \mu(x) K(x,y) .
\end{align*}
Using the detailed balance condition ensures that
$$
\sum_{\genfrac{}{}{0pt}{}{x,y \in V:}{f(y) \geq f(x)}} (f(x)-f(y))^2 e^{f(y)} \mu(x) K(x,y)
=
\sum_{\genfrac{}{}{0pt}{}{x,y \in V:}{f(y) \geq f(x)}} (f(x)-f(y))^2 e^{f(y)} \mu(y) K(y,x)
$$
which, after a change of variable, implies that
$$
\sum_{x,y \in V} (e^{f(y)} - e^{f(x)})(f(y)-f(x)) \mu(x) K(x,y)
=
2 \!\!\! \sum_{\genfrac{}{}{0pt}{}{x,y \in V:}{f(x) \geq f(y)}} \!\!\! (f(y)-f(x))^2 e^{f(x)} \mu(x) K(x,y).
$$
Now, we observe that
\begin{align*}
\sum_{\genfrac{}{}{0pt}{}{x,y \in V:}{f(x) \geq f(y)}} \!\!\! (f(y)-f(x))^2 e^{f(x)} & \mu(x) K(x,y) \\
& =
\sum_{\genfrac{}{}{0pt}{}{x,y \in V:}{f(x) \geq f(y)}} \!\!\! \left( \frac{[f(y)-f(x)]_-}{d(x,y)} \right)^2 e^{f(x)} \mu(x) K(x,y) d(x,y)^2 \\
& \leq
\sum_{x\in V} |\widetilde \nabla f|^2(x) e^{f(x)} \mu(x) \sum_{y \in V} K(x,y) d(x,y)^2
\end{align*}
which leads to the desired result since $\sum_{y \in V} K(x,y) d(x,y)^2 \leq L$. The proof is complete.
\end{proof}


%

\subsection{Hypercontractivity property of the family of operators $(\exp\{\tQ_t\})_{t \geq 0}$: proof of Corollary \ref{cormain} and Theorem\ref{equivalencelogsobtalagrand}}\label{sec}

Using our main result on the Hamilton-Jacobi equation, we shall follow the line of proof of
\cite{BGL} to prove Corollary \ref{cormain}, namely that the family of operator $(\exp\{\tQ_t\})_{t \geq 0}$ enjoys some hypercontractivity property. As a byproduct we shall prove that the modified log-Sobolev inequality \eqref{logsob} implies the transport-entropy inequality \eqref{t2}, giving rise, thanks to Proposition \ref{toto} to a variety of non trivial examples satisfying such an inequality, on graphs.

\begin{proof}[Proof of Corollary \ref{cormain}]
We shall show that the modified log-Sobolev inequality \eqref{eqMLS-(2/c)intro} implies the hypercontractivity property \eqref{eq7intro} for positive $\rho$ and the modified log-Sobolev inequality \eqref{eqMLS-(2/c)intro1} implies the hypercontractivity property \eqref{eq7intro} for negative $\rho$ at the same time. To that purpose, fix $\rho \in \RR$ and, following \cite{BGL}, define
$$
F(t):=\frac{1}{k(t)}\log\left(\int  e^{k(t)\tQ_tf}\,d\mu\right), \qquad t \geq 0
$$
with $k(t):=\rho+(t/2C)$.
By Theorem \ref{the1}, $F$ is differentiable at every point $t>0$ when $\rho \geq 0$ and every $t\in (0,-\rho C/2)$ when $\rho \leq0$. For such points,
it holds
$$
F'(t)=\frac{k'(t)}{k(t)^2} \frac{1}{\int  e^{k(t)\tQ_tf} \,d\mu}
\left( \ent_\mu\left(e^{k(t)\tQ_tf}\right)
+\frac{k(t)^2}{k'(t)} \int  e^{k(t)\tQ_tf} \frac{\partial}{\partial t}\tQ_tf \,d\mu
\right).
$$
According to Theorem \ref{thm:main}, we have
\begin{align*}
\ent_\mu\left(e^{k(t)\tQ_tf}\right)
& +\frac{k(t)^2}{k'(t)} \int  e^{k(t)\tQ_tf} \frac{\partial}{\partial t}\tQ_tf \,d\mu \\
& \leq
\ent_\mu\left(e^{k(t)\tQ_tf}\right)
- \frac{k(t)^2}{2k'(t)} \int  |\widetilde{\nabla}\tQ_tf|^2 e^{k(t)\tQ_tf}  \,d\mu \\
& =
\ent_\mu\left(e^{k(t)\tQ_tf}\right)
- \frac{1}{2k'(t)} \int   \left|\widetilde{\nabla} \left[|k(t)|\tQ_tf \right]\right|^2 e^{k(t)\tQ_tf} \,d\mu
\end{align*}
where the last equality follows from Remark \ref{remark}.
Now we have two cases to deal with:
$(a)$ If $\rho\geq 0$ and $\mu$ satisfies $\eqref{eqMLS-(2/c)intro}$, then $|k(t)|=k(t)$. Hence, applying
the modified log-Sobolev inequality \eqref{eqMLS-(2/c)intro} leads to $F'(t)\leq 0$.
$(b)$  If $\rho\leq 0$ and $\mu$ satisfies $\eqref{eqMLS-(2/c)intro1}$, then $|k(t)|=-k(t)$. Hence applying the modified log-Sobolev inequality \eqref{eqMLS-(2/c)intro1} leads also to $F'(t)\leq 0$.
In both cases $F'(t) \leq 0$ implies $F(t) \leq F(0)$ which amounts to \eqref{eq7intro}.


Conversely, suppose that \eqref{eq7intro} holds for every $t\geq0$ when $\rho > 0$ (respectively every $t\in [0, -\rho C/2)$ when $\rho < 0$) .
Then, in the limit, \eqref{eq7intro} implies that $F'(0)\leq 0$ and thus (recall that $k'(t) =1/(2C) >0$)
$$
\ent_\mu\left(e^{k(0)\tQ_0f}\right)
+\frac{k(0)^2}{k'(0)} \int  e^{k(0)\tQ_0f} \frac{\partial}{\partial t}\tQ_tf |_{t=0}  \,d\mu \leq 0
$$
where we set $\tQ_0f:= \lim_{t \to 0} \tQ_tf$. By Theorem \ref{thm:main2}, since $\alpha(x)=x^2/2$, $\tQ_0f=f$ so that the latter is equivalent to
$$
\ent_\mu\left(e^{\rho f}\right)
+2\rho^2C \int  e^{\rho f} \frac{\partial}{\partial t}\tQ_tf |_{t=0}  \,d\mu \leq 0 .
$$
Now,  according to Theorem \ref{thm:main2},
$\frac{\partial}{\partial t}\tQ_tf(x)|_{t=0} = -\frac{1}{2}|\widetilde{\nabla}f|^2(x)$, $x \in X$ so that
$$
\ent_\mu\left(e^{\rho f}\right)
-C  \int  e^{\rho f} |\widetilde{\nabla}(|\rho|f)|^2 \,d\mu \leq 0 .
$$
This precisely amounts to proving \eqref{eqMLS-(2/c)intro} (respectively \eqref{eqMLS-(2/c)intro1}) when $\rho \geq 0$ (resp. $\rho \leq 0$).
The proof of Corollary \ref{cormain} is complete.
\end{proof}

\medskip

\begin{proof}[proof of Theorem\ref{equivalencelogsobtalagrand}]
In order to prove $(i)\Rightarrow (ii)$ of Theorem \ref{equivalencelogsobtalagrand}, we need to recall the following generalization of Bobkov-Gotze dual characterization borrowed from \cite[Theorem 5.5]{Dualdiscrete}:

Inequality \eqref{t2tilde} holds if and only if for all bounded continuous function $\varphi \colon X \to \RR$ it holds
\begin{equation} \label{dualform}
\int  \exp\left\{ \frac{2}{C}\tQ_1 \varphi \right\} \,d\mu \leq \exp \left\{ \frac{2}{C} \int  \varphi \,d\mu \right\} .
\end{equation}

Now,  \eqref{eq7intro} applied to $\rho=0$ and $t=1$ precisely amounts to
\eqref{dualform}, since by definition $\| g \|_0:= \exp\{\int \log g \,d\mu \}$ for $g \geq0$. Hence the result, thanks to
the dual characterization of \cite{Dualdiscrete}.

Now we turn to prove $(ii)\Rightarrow (i)$.
According to  \cite[Proposition 8.3]{Dualdiscrete}, $(ii)$ implies that for all $\lambda\in (0,1/C_2)$, the following inequality holds for all bounded lower semi continuous function $f$:
$$\ent_\mu(e^f)\leq \frac{1}{1-\lambda C}\int (f-R_c^\lambda f)e^fd\mu.$$
Here in our settings, $R_c^\lambda f(x):=\inf_{p\in \mathcal{P}(X)}\{\int f dp + \frac{\lambda}{2}(\int d(x,.)dp)^2\}=\widetilde{Q}_{1/\lambda}f(x)$.
According to \cite[Proposition 2.2]{FS2015}, $t\mapsto \widetilde{Q}_tf$ is convex. Thus, combining with theorem \ref{thm:main}, it holds
$$R_c^\lambda f-f=\widetilde{Q}_{1/\lambda}f-f\geq \frac{1}{\lambda}\frac{\partial}{\partial t}\widetilde{Q}_tf|_{t=0}=-\frac{1}{2\lambda}|\widetilde\nabla f|^2.$$
We deduce that
\begin{align*}
 \ent_\mu(e^f)&\leq \frac{1}{1-\lambda C}\int (f-R_c^\lambda f)e^fd\mu\\
 &\leq \frac{1}{2\lambda(1-\lambda C)}\int |\widetilde\nabla f|^2e^fd\mu.
\end{align*}
Optimizing $\lambda$ with $\lambda=\frac{1}{2C}$ yields the result.
\end{proof}

%

\begin{Rem}
Since $|\widetilde \nabla f|^2(x) \leq 1$ for any $1$-Lipschitz function, the usual Herbst argument (see \textit{e.g.} \cite[Chapter 7]{1}, \cite{BobkovTetali}) applies and leads to the following concentration result: if $\mu$ satisfies
the modified log-Sobolev inequality \eqref{logsob}, then any $1$-Lipschitz function $f \colon X \to \RR$
with $\int  f\,d\mu=0$ satisfies $\mu(f\geq h)\leq e^{-h^2/(4C)}$ for all $h \geq 0$.
\end{Rem}

\subsection{Poincar\'e inequality}\label{Poincare}

In this section, we prove that the Poincar\'e inequality \eqref{poincare} is equivalent to the transport-entropy inequality \eqref{logsob} with a quadratic-linear cost, a notion we define below. This will extend to our setting similar results known in the continuous, see \cite{BGL}.

\begin{Def}[Quadratic-linear cost function]\label{quadratic-linear}
A quadratic-linear cost function $\alpha_a^h:\RR^+\rightarrow \RR$,  $a,h>0$ is such that
$$
\alpha_a^h(x)=
\begin{cases}
ax^2 & x\leq h \\2ax-ah^2 & x>h .
\end{cases}
$$
\end{Def}

The main theorem of this section is the following.

\begin{The}\label{poincareT2}
Let $\mu$ be a probability measure on $X$. The following propositions are equivalent.
\begin{description}
\item $(i)$ There exists a constant $C_1>0$ such that $\mu$ satisfies the Poincar\'e inequality \eqref{poincare} with constant $C_1$.
\item $(ii)$ There exist constants $C_2,a,h>0$ such that $\mu$ satisfies the weak transport-entropy inequality  \eqref{logsob} with constant $C_2$ and cost $\alpha_a^h$.
\end{description}
More precisely,
\begin{description}
\item{-} $(ii)$ implies $(i)$ with $C_1=aC_2$;
\item{-} $(i)$ implies $(ii)$ with $C_2=K(c)/2$, $a=\frac{1}{4K(c)}$ and $h=2cK(c)$ for any
 $c<2/\sqrt{C_1}$ and
$$
K(c):=\frac{C_1}{2} \left(\frac{2+2e^2+c\sqrt{C_1}}{2-c\sqrt{C_1}}\right)^2e^{c\sqrt{5C_1}}.
$$
\end{description}
\end{The}

\begin{Rem} \label{rem:poincare}
As a direct consequence of the above theorem, we observe that the weak transport-entropy inequality \eqref{t2} with cost function $\alpha(x):=\frac{x^2}{2}$ and constant $C$ implies the Poincar\'e inequality \eqref{poincare} with constant $C/2$. Indeed, since $\alpha(x)= \frac{x^2}{2}\geq \alpha_{1/2}^2(x)$, the weak transport-entropy inequality $\widetilde{T}_2(C)$ implies $\widetilde{T}_{\alpha_{1/2}^2}(C)$ and the conclusion follows from Item $(ii)$ of Theorem \ref{poincareT2}.
\end{Rem}

The proof of Theorem \eqref{poincareT2} relies on a characterization of the Poincar\'e Inequality
\eqref{poincare} in term of a modified log-Sobolev inequality with quadratic-linear cost, of independent interest.
Such a characterization is an extension of a well known result of Bobkov and Ledoux \cite{BL}.

\begin{The}\label{poincarelog}
A probability measure $\mu$ on $X$ satisfies the Poincar\'e Inequality \eqref{poincareT2}
with constant $C$ if and only if $\mu$ satisfies the modified log-Sobolev inequality \eqref{logsob} with constant $C'$ and cost $\alpha_a^h$. 
More precisely
\begin{description}
\item{-} \eqref{poincareT2}  implies \eqref{logsob} with $C'=K(c)$, $a=\frac{1}{4K(c)}$ and $h=2cK(c)$ for any $c<2/\sqrt{C}$ with $K(c)$ defined in theorem \ref{poincareT2};
\item{-} \eqref{logsob} implies \eqref{poincareT2} with $C=C'$.
\end{description}
\end{The}

We observe that, with respect to \cite{BL} there is a loss in the constant $K(c)$. This is technical. Indeed, the proof of Bobkov and Ledoux cannot be extended directly and one has to be careful in many points. Since the proof of
Theorem \ref{poincarelog} deals only with properties of $\widetilde \nabla$ and not with the Hamilton-Jacobi equation, and because it is long and technical, we decided to postpone it to the appendix.

\begin{proof}[Proof of Theorem \ref{poincareT2}]
We will first prove that $(i)$ implies $(ii)$.
Fix $c <2/\sqrt{C}$ and set $C=C_1$, $a=\frac{1}{4K(c)}$ and $h=2cK(c)$.
Thanks to Theorem \ref{poincarelog} for all $f \colon X \to \RR$ bounded, it holds
\begin{equation*}
\ent_\mu(e^f)\leq K(c)\int  (\alpha_a^{h})^*(|\widetilde{\nabla}f|)e^f\,d\mu .
\end{equation*}
Arguing  as in the proof of Corollary \ref{cormain} (see Section\ref{sec}) with $k(t)=2t/K(c)$, and using the
fact\footnote{For the reader convenience we observe that
$(\alpha_a^{h})^*(x)=K(c)x^2$ if   $|x|\leq c$ and  $(\alpha_a^{h})^*(x)=+\infty$ otherwise.} that $(\alpha_a^{h})^*(\lambda u)\leq \lambda^2 (\alpha_a^{h})^*(u)$ as soon as $u\leq 2ah$, we obtain (details are left to the reader) that the family of operators $(\exp\{ \widetilde Q_t\})_{t \geq 0}$, with $\widetilde Q$ defined with the cost $\alpha_a^h$, is hypercontractive which in turn guarantees that
$$
\int  \exp\left\{ \frac{2}{K(c)}\tQ_1 f \right\} d\mu \leq \exp \left\{ \frac{2}{K(c)} \int  f \,d\mu \right\}
$$
for all bounded function $f$. The conclusion follows from the dual characterization of \cite{Dualdiscrete} (that we recalled in \eqref{dualform}).

Next we prove that  $(ii)\Rightarrow (i)$.
By an easy argument it is enough to prove \eqref{poincare} for all bounded Lipschitz function $f$ on $X$. According to \cite{Dualdiscrete} (see \eqref{dualform}), the transport-entropy inequality \eqref{logsob}, with cost $(\alpha_a^h)^*$, is equivalent to say that for all continuous bounded function $\phi$ on $X$
 it holds
$$
\int \exp\left\{\frac{2}{C_2}\tQ_1\phi\right\}d\mu\leq \exp\left\{\int  \frac{2}{C_2}\phi \,d\mu\right\}
$$
where $\widetilde Q$ is defined with the cost $\alpha_a^h$.
Fix $l>0$, let $f$ be a $l$-Lipschitz function and set $\phi:=tf$. The latter inequality reduces to
$\int \exp\left\{\frac{2}{C_2}\tQ_1tf\right\}d\mu\leq \exp\left\{\int  \frac{2}{C_2}tf \,d\mu\right\}$.
Hence, for $t<(ah)/l$, by Lemma \ref{lemQ1Qt} below, we get
$$
\int \exp\left\{\frac{2}{C_2}t\tQ_tf\right\}\,d\mu\leq \exp\left\{\int  \frac{2}{C_2}tf \,d\mu\right\}.
$$
An expansion around $t=0$ yields that
\begin{multline}
 \int  \left(1+\frac{2}{C_2}tf+\frac{1}{2}t^2\left(\frac{4}{C_2^2}f^2+\frac{4}{C_2}\frac{\partial}{\partial t}\tQ_tf_{|_{t=0}}\right)+o(t^2)\right)d\mu \\
 \leq 1+t\frac{2}{C_2}\int  f\,d\mu+\frac{1}{2}t^2\frac{4}{C_2^2}\int  f\,d\mu+o(t^2).
\end{multline}
Therefore (comparing the coefficients of $t^2$), it holds
$\var_\mu(f)\leq -C_2\int  \frac{\partial}{\partial t}\tQ_tf_{|_{t=0}}d\mu$. Applying Theorem \ref{thm:main2} we arrive at $\var_\mu(f)\leq C_2\int  \alpha_a^{h*}\left(|\widetilde \nabla f|\right)\, d\mu$, which in turn,
%
since $\alpha_a^{h*}\left(|\widetilde \nabla f|(x)\right)=a\left(|\widetilde \nabla f|(x)\right)^2$ for $l\leq ah$,  implies that for all $ah$-Lipschitz function $f$, it holds
 $$
 \var_\mu(f)\leq aC_2\int  |\widetilde \nabla f|^2\,d\mu.
 $$
 Replacing $f$ by $\lambda f$ with $\lambda \in \RR^+$, we conclude that the above inequality holds for all Lipschitz function $f$ and thus $\mu$ satisfies the Poincar\'e inequality with constant $aC_2$.
This ends the proof of the theorem.
\end{proof}

\begin{Lem}\label{lemQ1Qt}
Let $f$ be an $l$-Lipschitz function and $\tQ_t$ be the inf-convolution for a quadratic-linear cost function $\alpha_{a}^{h}$, $a,h >0$.
Then, for all $x\in X$ and all $t< (ah)/l$, it holds
$\tQ_1(tf)(x)=t\tQ_tf(x)$.
\end{Lem}

\begin{proof}
Fix $t<ah/l$ and $x \in X$.
For all $p\in m_{tf}(1,x)$ (defined in \eqref{defm}) we have by Item $(i)$ of Theorem \ref{the1}
$$
\int  tf(y)\,p(dy)+\alpha_a^h\left(\int d(x,y)\,p(dy)\right) = \widetilde Q_1 (tf)(x) \leq tf(x) .
$$
Hence
\begin{align*}
\alpha_a^h\left(\int d(x,y)\,p(dy)\right)
&\leq
t\int f(x)-f(y)\,p(dy)\leq t|\widetilde \nabla f|(x)\int d(x,y)\,p(dy)\\
&\leq
tl\int d(x,y)\,p(dy)\leq ah\int d(x,y)\,p(dy),
\end{align*}
where we used that  $f(x)-f(y) \leq |\widetilde \nabla f|(x)d(x,y)$ and the fact that $f$ is $l$-Lipschitz.
Since for quadratic-linear cost  $\alpha_a^h(u)\leq ahu$ if and only if $u\leq h$, the above inequality implies that  $\int d(x,y)\,p(dy)\leq h$ and that $\alpha_a^h\left(\int d(x,y)\,p(dy)\right)=a\left(\int d(x,y)\,p(dy)\right)^2$. Therefore
$$
\tQ_1(tf)(x)=\inf_{p\in \mathcal{P}(X)}\left\{\int  tf\,dp+a\left(\int d(x,y)\,p(dy)\right)^2\right\}.
$$
Similarly  for all $q\in m_{f}(t,x)$ it holds
$$
\int  f(y)\,\,q(dy)+t\alpha_a^h\left(\frac{\int d(x,y)\,q(dy)}{t}\right) \leq f(x) .
$$
Therefore
\begin{align*}
\alpha_a^h\left(\frac{\int d(x,y)\,q(dy)}{t}\right)
&\leq
\int f(x)-f(y)\,q(dy)
\leq
|\widetilde \nabla f|(x)\int d(x,y)\,q(dy)\\
&\leq
l\int d(x,y)\,p(dy)
\leq
\frac{ah}{t}\int d(x,y)\,q(dy).
\end{align*}
This (due to the specific shape of the quadratic-linear cost) leads to
$\int d(x,y)\,q(dy)/t\leq h$ and $\alpha_a^h\left(\frac{\int d(x,y)\,q(dy)}{t}\right)=a\left(\frac{\int d(x,y)\,q(dy)}{t}\right)^2$. Therefore,
$$
\tQ_t f(x)=\inf_{p\in \mathcal{P}(X)}\left\{\int  f\,dp+ \frac{a}{t}\left(\int d(x,y)\,p(dy)\right)^2\right\}.
$$
As a conclusion,
\begin{align*}
t\tQ_tf(x)
&=
t\inf_{q\in \mathcal{P}(X)}\left\{\int  f\,dq+\frac{a}{t}\left(\int d(x,y)\,p(dy)\right)^2\right\}\\
&=
\inf_{p\in \mathcal{P}(X)}\left\{\int  tf\,dp+a\left(\int d(x,y)\,p(dy)\right)^2\right\}=\tQ_1(tf)(x).
\end{align*}
\end{proof}

\section{Examples} \label{sec:examples}

In this section, we give some examples of application. In  particular, we shall see that our theorems are optimal in many situations. More precisely the first two examples deal with equality versus strict inequality in Theorem
\ref{thm:main2}. The other examples are more concerned with functional inequalities.


\subsubsection*{Example of $\RR^n$, equality case}

Let $\alpha(x)=x^2/2$, $x\in \RR^+$ and $f \colon \RR^n \to \RR$ convex. Then for all $t\geq 0$,
\[
\frac{\partial}{\partial t}\tQ_tf(x)+ \frac{1}{2} |\widetilde{\nabla}\tQ_tf|^2(x) = 0,
\]
\textit{i.e.} there is actually equality in Item $(i)$ of Theorem \ref{thm:main2}.

To prove this fact, we observe first that, since $\lim_{h\rightarrow \infty} \alpha'(h)= \infty$, the thesis follows from Item $(ii)$ of Theorem \ref{thm:main2}  when $t=0$.
For $t>0$,  since $f$ is convex, Proposition \ref{exconvex} ensures that $\tQ_tf=Q_tf$. Moreover, for all convex function $f$, $Q_tf$ is a convex function which guarantees that
$|\widetilde{\nabla}Q_tf|=|\nabla Q_tf|$ (where $|\nabla \cdot|$ is the Euclidean length of the usual gradient). Hence, the claim follows from the classical Hamilton-Jacobi equation that precisely asserts that for $t>0$, $\frac{\partial}{\partial t}Q_tf(x)+ \frac{1}{2} |\nabla Q_tf|^2(x) = 0$.


\subsubsection*{Example of the two points space $\{0,1\}$, strict inequality case}

Let $\alpha(x)=x^2/2$ and $X=\{0,1\}$ (the graph consisting of two points). Consider $f$ such that $f(0)=1$ and $f(1)=0$. It is easy to see that for $t\in(0,1)$, $\tQ_tf(0)=1-\frac{t}{2}$ and $\tQ_tf(1)=0$. It leads to $|\widetilde{\nabla}\tQ_tf|(0)=1-\frac{t}{2}$ and $\frac{\partial}{\partial t}\tQ_tf(0)=-\frac{1}{2}$. Thus, for all $t\in(0,1)$, $\frac{\partial}{\partial t}\tQ_tf(0)+ \frac{1}{2} |\widetilde{\nabla}\tQ_tf|^2(0) <0 $,
\textit{i.e.} the inequality in Item $(i)$ of Theorem \ref{thm:main2} is strict.
We observe that, more generally, the same conclusion holds as soon as $X$ has at least one isolated point $x_o$ (take $f$ with $f(x_o)=0$ and $f(y)=1$  for all $y\neq x_o$).

\bigskip

Next we give examples of measures satisfying log-Sobolev/Poincar\'e/transport-entropy type inequalities.

\subsubsection*{Measures satisfying the log-Sobolev inequality \eqref{t2tilde} and the transport-entropy \eqref{t2tilde}}

As already mentioned, the classical log-Sobolev inequality \eqref{gross} implies the (say) classical modified log-Sobolev inequality  \eqref{usual-logsob} which, thanks to Proposition \ref{toto} implies under mild assumptions the modified log-Sobolev inequality \eqref{logsob}, which finally, thanks to Corollary \ref{cormain}, implies the transport-entropy inequality \eqref{t2tilde}.
The latter is usually hard to obtain directly. The above chain of implication applies to a lot of different situations, including highly non-trivial examples. Let us mention random walks on the hypercube, on the symmetric group or the complete graph (see \cite{BobkovTetali} where optimal (or almost optimal) bounds are given for \eqref{usual-logsob}) the optimal bound in  \eqref{gross} for the lamplighter graph can be found in \cite{Abakoumov}, and in \cite{martinelli2} for the Ising model at high temperature, on the lattice or on trees.
Many other examples can be found in \cite{diaconis}... Bound on the constant in the tranport-entropy inequality \eqref{t2tilde} are new for all examples listed above, to the best of our knowledge.

\medskip

As an illustration, consider the uniform measure $\mu \equiv 1/2^n$ on the hypercube $\{0,1\}$ associated to the Markov chain that jumps from $x$ to anyone of its nearest neighbors (\textit{i.e.} any string $x'$ that differs from $x$ in exactly one coordinate) with equal probability ($1/n$). Then $\mu$ satisfies Gross' Inequality \eqref{gross} with constant $n/2$ \cite{G75}, the classical modified log-Sobolev inequality
\eqref{usual-logsob} with constant $n/8$ \cite{BobkovTetali}, and thus, by Proposition \ref{toto} (note that $L=1$), the modified log-Sobolev inequality \eqref{t2tilde} with constant $n/4$, and in turn, thanks to Corollary \ref{cormain}, the transport-entropy inequality \eqref{t2tilde} holds with constant $n/8$.

\medskip

In the case of the symmetric group $S_n$, consisting of $n!$ permutation (of $n$ elements), equipped with the transposition distance (\textit{i.e.}\ two permutations are at distance 1 if one is the other composed with a transposition). Each permutation has $n(n-1)/2$ neighbors and the Markov chain that jumps uniformly at random to any neighbor is reversible with respect to the uniform measure $\mu \equiv 1/n!$. Gross' Inequality is known to hold with a constant of order $n^3 \log n$ \cite{lee-yau}, while the classical modified log-Sobolev inequality
\eqref{usual-logsob} holds with constant $C \leq n(n-1)^2/2$ \cite{BobkovTetali}. Therefore,
 by Proposition \ref{toto} (again note that $L=1$), $\mu$ satisfies the modified log-Sobolev inequality \eqref{t2tilde} with constant $n(n-1)^2$ and in turn, thanks to Corollary \ref{cormain}, the transport-entropy inequality \eqref{t2tilde} with constant $n(n-1)^2/2$.

\subsubsection*{Poincar\'e inequality}

The next proposition extends a well-known result that asserts that the Poincar\'e inequality holds on bounded domains. We will then give examples of measures satisfying the Poincar\'e inequality \eqref{poincare} but not the
one with the usual gradient.

\begin{Pro}\label{propoincarebornee}
Assume that the support of the probability measure $\mu$ has a finite diameter  and let $D=\sup_{x,y \in \mathrm{Supp}(\mu)} \{d(x,y)\}$. Then $\mu$ satisfies the Poincar\'e Inequality \eqref{poincare} with constant at most $D^2/2$.
\end{Pro}

\begin{proof}
For all $x,y\in \mathrm{Supp}(\mu)$, $f(x)-f(y)\leq d(x,y) |\widetilde \nabla f|(x)\leq D |\widetilde \nabla f|(x)$. Thus, for all continuous function $f$ on $X$, it holds
\begin{align*}
\var_\mu(f)
&=
\frac{1}{2} \iint_{\mathrm{Supp}(\mu)^2} \left(f(x)-f(y)\right)^2\mu(dx) \mu(dy)
\leq \frac{D^2}{2} \int  |\widetilde \nabla f|^2 \,d\mu.
\end{align*}
\end{proof}



Now, on $X=\RR$ consider the following probability measure  $\mu=\frac{1}{2}\delta_0+\frac{1}{2}\delta_1$.
We claim that $\mu$ satisfies the Poincar\'e inequality \eqref{poincare}, but not the (classical) Poincar\'e inequality with the Euclidean gradient.

Indeed, Proposition \ref{propoincarebornee} applies and leads to the Poincar\'e inequality \eqref{poincare}
with constant at most $1/2$. On the other hand,
the mapping  $f \colon \RR \ni x  \mapsto 2x^3-3x^2+1$ satisfies $f(0)=1$, $f(1)=0$ and $f'(0)=f'(1)=0$ so that
$\var_\mu (f)=\frac{1}{4}\left(f(0)-f(1)\right)^2=\frac{1}{4}$ and $\int  f'^2\,d\mu=0$ which proves the claim.

Let us prove now that $\mu$ also satisfies the modified log-Sobolev inequality \eqref{eqMLS-(2/c)intro1}.
Given $f \colon \RR \to \RR$ with $f(0) \geq f(1)$ (the other direction is similar), we have
$f(0)-f(1) \leq |\widetilde \nabla f|(0)$ so that
$\left(f(0)-f(1)\right)^2e^{f(0)} \leq \int_\RR |\widetilde \nabla f|^2 e^f\,d\mu$.  Thus,
to prove that the modified log-Sobolev inequality \eqref{eqMLS-(2/c)intro1} holds, it is enough to prove the existence of a constant $C$ such that
$$
\ent_\mu(f)\leq C\left(f(0)-f(1)\right)^2e^{f(0)}
$$
or equivalently
$$
f(0)e^{f(0)}+f(1)e^{f(1)}-\left(e^{f(0)}+e^{f(1)}\right)\log\left(\!\frac{e^{f(0)}+e^{f(1)}}{2}\!\right)
\leq
2C\left(f(0)-f(1)\right)^2e^{f(0)}.
$$
Setting $u:=f(0)-f(1) \geq 0$, the latter is equivalent to prove that
$$
ue^u-(e^u+1)\log\left(\frac{e^u+1}{2}\right)\leq Cu^2e^u \qquad \forall u \geq 0
$$
which is an easy exercise.


%

\section*{Appendix}

In this appendix we prove Theorem \ref{poincarelog}. The proof essentially follows \cite{BL}. However, many points in the original proof of Bobkov and Ledoux need to be adjusted, for technical reasons coming from the gradient $\widetilde \nabla$.

The proof relies on the following three propositions.

\begin{Pro}\label{proef}
If $\mu$ satisfies the Poincar\'e inequality \eqref{poincare} with constant $C>0$, then for all $f:X\mapsto \RR$,
$$\var_{\mu}(fe^{f/2})\leq C\int |\widetilde \nabla f|^2\left(1+e^4+f+\frac{f^2}{4}\right)e^f\,d\mu.$$
\end{Pro}

\begin{Pro}\label{propBL1}
If $\mu$ satisfies the Poincar\'e inequality \eqref{poincare} with constant $C>0$, then for any bounded $c$-Lipschitz function $f$ on $X$ with $c<2/\sqrt{C}$ and $\int  f \,d\mu=0$,
$$\int  f^2e^f \,d\mu \leq C \left(\frac{2+2e^2+c\sqrt{C}}{2-c\sqrt{C}}\right)^2\int  |\widetilde{\nabla}f|^2e^f\,d\mu.$$
\end{Pro}

\begin{Pro}\label{propBL}
If $\mu$ satisfies the Poincar\'e inequality \eqref{poincare} with constant $C>0$, then for any bounded function $f$ on $X$ with $\|f\|_{\mathrm{Lip}}\leq c$ and $\int  f\,d\mu=0$, we have
$$\int f^2\,d\mu \leq e^{c\sqrt{5C}}\int f^2e^{-|f|}\,d\mu.$$
\end{Pro}

We postpone the proof of the above propositions to prove Theorem \ref{poincarelog}.
\begin{proof}[Proof of Theorem \ref{poincarelog}]
Changing $f$ into $f+\text{constant}$ we may assume that $\int f\,d\mu=0$. Since
$u\log u\geq u-1$ for all $u\geq 0$, we have
$$
\ent_\mu(e^f)\leq \int (fe^f-e^f+1)\,d\mu=\int \left(\int_0^1tf^2e^{tf}\,dt\right)d\mu.
$$
Let $\phi(t):=\int f^2e^{tf}\,d\mu$, $t \in [0,1]$. By convexity, $\phi$ attains its maximum at either $t=0$ or $t=1$. By Proposition \ref{propBL}, and since $e^{-|f|}\leq e^f$, $\phi(0)\leq e^{c\sqrt{5C}}\phi(1)$. Thus, for every $t\in [0,1]$, $\phi(t)\leq e^{c\sqrt{5/C}}\phi(1)$. It follows that
$$\ent_\mu(e^f)\leq \int_0^1t\phi(t)\,dt\leq \int_0^1te^{c\sqrt{5C}}\phi(1)\,dt=\frac{1}{2}e^{c\sqrt{5C}}\int f^2e^{f}\,d\mu.$$
Together with Proposition \ref{propBL1}, Theorem \ref{poincarelog} is established.
\end{proof}

Now let us prove Propositions \ref{proef}, \ref{propBL1} and \ref{propBL}.

\begin{proof}[Proof of Proposition \ref{proef}]
Let $G:=u \mapsto ue^{u/2}$ and observe that it is decreasing on $(-\infty,-2]$, increasing on $(-2,\infty)$ and its minimum is $G(-2)=-2e^{-1}$. Now, starting from $G$, define an increasing function $H$ as $G$ when $G$ is increasing and as the symmetric of $G$ with respect to $y=G(-2)$ when $G$ is non-increasing. More precisely,
$$
H(x):=\begin{cases} -xe^{x/2}-4e^{-1} & \mbox{if } x\leq -2 \\ xe^{x/2} & \mbox{if } x>-2 . \end{cases}
$$
Observe that $|H(x)+2e^{-1}|=xe^{x/2}+2e^{-1}$, $x \in \RR$. Hence, using
that $\var_\mu(|g|)\leq \var_\mu(g)$, it holds
$$
\var_\mu (fe^{f/2})=\var_\mu(fe^{f/2}+2e^{-1})\leq \var_\mu(H\circ f+2e^{-1})=\var_\mu(H\circ f).
$$
Now applying the Poincar\'e Inequality \eqref{poincare} and Proposition \ref{propnabla} we have
\begin{equation}\label{eqH}
\var_\mu(H\circ f)
\leq
C\int  |\widetilde \nabla (H\circ f)|^2\,d\mu
\leq
C\int |\widetilde \nabla f|^2|\widetilde \nabla H|^2\left(f\right)d\mu .
\end{equation}
Since $H$ is increasing, we have
$$
0
\leq
|\widetilde \nabla H|(u)
=
\sup_{v<u}\frac{H(u)-H(v)}{u-v}
=\sup_{v<u}\left\{
\frac{1}{u-v}\int_{(v,u)}H'(t)\,dt\right\}
\leq
\sup_{t<u}H'(t).
$$
After some basic analysis, we have the following facts
\begin{itemize}
\item if $u<-4$, $\sup_{t<u}H'(t)=|(1+u/2)e^{u/2}|$ since $H'(t) =|(1+t/2)e^{t/2}|$ is increasing on $(-\infty,-4]$;
\item if $u\in [-4,0]$, $\sup_{t<u}H'(t) \leq 1\leq e^2e^{u/2}$;
\item if $u>0$, $\sup_{t<u}H'(t)=|(1+u/2)e^{u/2}|$ since $H'$ is increasing on $[0,\infty)$ and $H'(u)>H(0)=1\geq \sup_{t\leq 0}H'(t)$.
\end{itemize}
As a consequence, we have $|\widetilde \nabla H|^2(u)\leq \left((1+u/2)^2+e^4\right)e^{u}$.
Therefore
\begin{align*}
\var_\mu(H\circ f)
\leq
C\int |\widetilde \nabla f|^2|\widetilde \nabla H|^2\left(f\right)\,d\mu \leq
C\int |\widetilde \nabla f|^2\left(1+e^4+f+\frac{f^2}{4}\right)e^{f}\,d\mu.
\end{align*}
This ends the proof of the proposition.
\end{proof}

\begin{proof}[Proof of Proposition \ref{propBL1}]
Set $a^2=\int  f^2e^f\,d\mu$ and $b^2=\int  |\widetilde{\nabla}f|^2e^f\,d\mu$. By the Poincar\'e inequality \eqref{poincare}, for any two bounded functions $g$ and $h$ on $X$ with $\int  g\,d\mu=0$,
\begin{align*}
\left(\int  gh\,d\mu\right)^2
\leq
\left(\int  g^2\,d\mu\right)\left(\int  h^2\,d\mu\right)
\leq
\left(C\int |\widetilde{\nabla}g|^2\,d\mu\right)\left(C\int  |\widetilde{\nabla} h|^2\right)\,d\mu.
\end{align*}
Therefore, since $\int f\,d\mu=0$,
$$\left(\int fe^{f/2}\,d\mu\right)^2\leq C^2\left(\int |\widetilde{\nabla}f|^2\,d\mu\right)\left(\int  |\widetilde{\nabla}e^{f/2}|^2\,d\mu\right).$$

Set $G(u)=e^{u/2}$, $u \in \RR$. The convexity of $G$ guarantees that $|\widetilde \nabla G|=|G'|$. Thus by Proposition \ref{propnabla}, it holds  $|\widetilde{\nabla}e^{f/2}|^2\leq \frac{1}{4}|\widetilde{\nabla}f|^2e^{f}$.
Hence
$$
\left(\int  fe^{f/2}\,d\mu\right)^2\leq \frac{1}{4}C^2c^2b^2.
$$
On the other hand, according to Proposition \ref{proef},
\begin{align*}
\var_\mu (fe^{f/2})
& \leq
C\int |\widetilde \nabla f|^2\left(1+e^4+f+\frac{f^2}{4}\right)e^f\,d\mu \\
& \leq
C\left((1+e^4)b^2+\int |\widetilde \nabla f|^2fe^fd\mu+\frac{c^2a^2}{4}\right) .
\end{align*}
By Cauchy-Schwarz' Inequality,
$$
\int  |\widetilde{\nabla}f|^2fe^f\,d\mu \leq \left(\int  |\widetilde{\nabla}f|^2f^2e^{f}\,d\mu\right)^{1/2}\left(\int  |\widetilde{\nabla}f|^2e^{f}\,d\mu\right)^{1/2}\leq cab,$$
so that
$$ \var_\mu (fe^{f/2})\leq C\left(\left(b+\frac{ca}{2}\right)^2+e^4b^2\right).$$
Then we get that
$$a^2=\left(\int fe^{f/2}\,d\mu\right)^2+\var_\mu (fe^{f/2})\leq \frac{1}{4}C^2c^2b^2+C\left(b+\frac{ca}{2}\right)^2 +Ce^4b^2.$$
Simplifying this inequality, we end up with
$$
\frac{a}{b}\leq \sqrt{C} \left(\frac{2+2e^2+c\sqrt{C}}{2-c\sqrt{C}}\right),
$$
and the conclusion follows.
\end{proof}

\begin{proof} [Proof of Proposition \ref{propBL}]
For all $u>0$ and all $v\in \RR$, we have $2|v|\leq u+(1/u)v^2$. Hence $2|v|^3\leq uv^2+(1/u)v^4$ and therefore,
\begin{equation}\label{1eq4}
2\int |f|^3\,d\mu\leq u\int f^2\,d\mu+\frac{1}{u}\int f^4\,d\mu.
\end{equation}
By the Poincar\'e inequality  \eqref{poincare} it holds
$$
\int  f^2\,d\mu\leq C\int   |\widetilde{\nabla}f|^2\,\mu(dx)\leq c^2C,
$$
so that $\left(\int f^2\,d\mu\right)^2\leq c^2C\int f^2\,d\mu$.

On the other hand, set $G(t)= t^2$, $t \geq 0$. The convexity of $G$ guarantees that for all $t\geq 0$,
$|\widetilde \nabla G|(t)=|G'|(t)$. Hence, according to Proposition \ref{propnabla}, it holds
\begin{align*}\label{1eq3}
\var_\mu(f^2)
&=
\var_\mu(|f|^2)\leq C\int |\widetilde{\nabla}(|f|^2)|^2d\mu
\leq
4C\int f^2 |\widetilde{\nabla}|f||^2d\mu
\leq
4c^2C\int f^2d\mu
\end{align*}
where in the last inequality we used  that $|f|$ is $c$-Lipschitz.
It follows that $\int f^4\,d\mu =(\int f^2\,d\mu)^2+\var_\mu(f^2) \leq 5c^2C\int f^2\,d\mu$. Hence, from (\ref{1eq4}), we obtain that for every $u>0$,
$$2\int |f|^3\,d\mu\leq \left(u+\frac{5c^2C}{u}\right)\int f^2\,d\mu.$$
Minimizing over $u>0$, we get
\begin{equation}\label{1eq5}
\int |f|^3\,d\mu \leq c\sqrt{5C}\int f^2\,d\mu.
\end{equation}
Consider now the probability measure $\tau(dx)=f(x)^2\,\mu(dx)/(\int  f^2\,d\mu)$. By Jensen's inequality,
$$\int f^2e^{-|f|}\,d\mu=\int e^{-|f|}\,d\tau \int |f|^2\,d\mu\geq e^{-\int f\,d\tau}\int |f|^2\,d\mu.  $$
By (\ref{1eq5}) we conclude that
$$\int |f|\,d\tau=\frac{\int |f|^3\,d\mu}{\int f^2\,d\mu}\leq c\sqrt{5C},$$
from which the result follows.
\end{proof}

\section*{Acknowledgement}
I warmly thank my PhD advisers Natha\"el Gozlan and Cyril Roberto for helpful advises and remarks.

%

\begin{thebibliography}{10}

\bibitem{Abakoumov}
E.~Abakumov, A.~Beaulieu, F.~Blanchard, M.~Fradelizi, N.~Gozlan, B.~Host,
  T.~Jeantheau, M.~Kobylanski, G.~Lecu{\'e}, M.~Martinez, M.~Meyer,
  M.~Mourgues, F.~Portal, F.~Ribaud, C.~Roberto, P.~Romon, J.~Roth, P.-M.
  Samson, P.~Vandekerkhove, and A.~Youssfi.
\newblock The logarithmic {S}obolev constant of the lamplighter.
\newblock {\em J. Math. Anal. Appl.}, 399(2):576--585, 2013.

\bibitem{ambrosio}
L.~Ambrosio, N.~Gigli, and G.~Savar{\'e}.
\newblock Calculus and heat flow in metric measure spaces and applications to
  spaces with {R}icci bounds from below.
\newblock {\em Invent. Math.}, 195(2):289--391, 2014.

\bibitem{1}
C.~An\'e, S.~Blach\`ere, D.~Chafai, P.~Foug\`eres, I.~Gentil, F.~Malrieu,
  C.~Roberto, and G.~Scheffer.
\newblock {\em Sur les in\'egalit\'es de Sobolev logarithmiques}, volume~10 of
  {\em Panoramas et Syth\`ese}.
\newblock Soci\'et\'e Math\'ematique de France, Paris, 2000.

\bibitem{balogh}
Z.~M. Balogh, A.~Engulatov, L.~Hunziker, and O.~E. Maasalo.
\newblock Functional inequalities and {H}amilton--{J}acobi equations in
  geodesic spaces.
\newblock {\em Potential Anal.}, 36(2):317--337, 2012.

\bibitem{barbu}
V.~Barbu and G.~Da~Prato.
\newblock {\em Hamilton-{J}acobi equations in {H}ilbert spaces}, volume~86 of
  {\em Research Notes in Mathematics}.
\newblock Pitman (Advanced Publishing Program), Boston, MA, 1983.

\bibitem{BGL}
S.~Bobkov, I.~Gentil, and M.~Ledoux.
\newblock Hypercontractivity of hamilton-jacobi equations.
\newblock {\em J.Math.Pures Appl.}, 80(7):669--696, 2001.

\bibitem{BL}
S.~Bobkov and M.~Ledoux.
\newblock Poincar\'e's inequalities and talagrand's concentration phenomenon
  for the exponential distribution.
\newblock {\em Probab.Theory Relat. Fields}, (107):383--400, 1997.

\bibitem{BobkovTetali}
S.~Bobkov and P.~Tetali.
\newblock Modified logarithmic sobolev inequalities in discrete settings.
\newblock {\em Journal of Theoretical Probability}, 19(2):289--335, 2006.

\bibitem{camilli}
F.~Camilli and C.~Marchi.
\newblock A comparison among various notions of viscosity solution for
  {H}amilton-{J}acobi equations on networks.
\newblock {\em J. Math. Anal. Appl.}, 407(1):112--118, 2013.

\bibitem{Dem97}
A.~Dembo.
\newblock Information inequalities and concentration of measure.
\newblock {\em Ann. Probab.}, 25(2):927--939, 1997.

\bibitem{diaconis}
P.~Diaconis and L.~Saloff-Coste.
\newblock Logarithmic {S}obolev inequalities for finite {M}arkov chains.
\newblock {\em Ann. Appl. Probab.}, 6(3):695--750, 1996.

\bibitem{evans}
L.~C. Evans.
\newblock {\em Partial differential equations}, volume~19 of {\em Graduate
  Studies in Mathematics}.
\newblock American Mathematical Society, Providence, RI, 1998.

\bibitem{FS2015}
M.~Fathi and Y.~Shu.
\newblock Curvature and transport inequalities for markov chains in discrete
  spaces.

\bibitem{gangbo}
W.~{Gangbo} and A.~{Swiech}.
\newblock Metric viscosity solutions of hamilton-jacobi equations.
\newblock Preprint, 2014.

\bibitem{GL13}
Nicola Gigli and Michel Ledoux.
\newblock From log {S}obolev to {T}alagrand: a quick proof.
\newblock {\em Discrete Contin. Dyn. Syst.}, 33(5):1927--1935, 2013.

\bibitem{gozlan09}
N.~Gozlan.
\newblock A characterization of dimension free concentration in terms of
  transportation inequalities.
\newblock {\em Ann. Probab.}, 37(6):2480--2498, 2009.

\bibitem{GRS14}
N.~Gozlan, C.~Roberto, and P.-M. Samson.
\newblock Hamilton {J}acobi equations on metric spaces and transport entropy
  inequalities.
\newblock {\em Rev. Mat. Iberoam.}, 30(1):133--163, 2014.

\bibitem{GRST12}
N.~Gozlan, C.~Roberto, P.-M. Samson, and P.~Tetali.
\newblock Displacement convexity of entropy and related inequalities on graphs.
\newblock {\em Probab. Theory Related Fields}, 160(1-2):47--94, 2014.

\bibitem{GRST14}
N.~Gozlan, C.~Roberto, P.-M. Samson, and P.~Tetali.
\newblock Kantorovich duality for general transport costs and applications.
\newblock 2014.

\bibitem{G75}
L.~Gross.
\newblock Logarithmic {S}obolev inequalities.
\newblock {\em Amer. J. Math.}, 97(4):1061--1083, 1975.

\bibitem{guionnet}
A.~Guionnet and B.~Zegarlinski.
\newblock Lectures on logarithmic {S}obolev inequalities.
\newblock In {\em S\'eminaire de {P}robabilit\'es, {XXXVI}}, volume 1801 of
  {\em Lecture Notes in Math.}, pages 1--134. Springer, Berlin, 2003.

\bibitem{HUL}
Jean-Baptiste Hiriart-Urruty and Claude Lemar{\'e}chal.
\newblock {\em Fundamentals of convex analysis}.
\newblock Grundlehren Text Editions. Springer-Verlag, Berlin, 2001.
\newblock Abridged version of {{\i}t Convex analysis and minimization
  algorithms. I} [Springer, Berlin, 1993; MR1261420 (95m:90001)] and {{\i}t II}
  [ibid.; MR1295240 (95m:90002)].

\bibitem{lee-yau}
T.-Y. Lee and H.-T. Yau.
\newblock Logarithmic {S}obolev inequality for some models of random walks.
\newblock {\em Ann. Probab.}, 26(4):1855--1873, 1998.

\bibitem{LV07}
J.~Lott and C.~Villani.
\newblock Hamilton-{J}acobi semigroup on length spaces and applications.
\newblock {\em J. Math. Pures Appl. (9)}, 88(3):219--229, 2007.

\bibitem{LV09}
J.~Lott and C.~Villani.
\newblock Ricci curvature for metric-measure spaces via optimal transport.
\newblock {\em Ann. of Math. (2)}, 169(3):903--991, 2009.

\bibitem{martinelli}
F.~Martinelli.
\newblock Lectures on {G}lauber dynamics for discrete spin models.
\newblock In {\em Lectures on probability theory and statistics
  ({S}aint-{F}lour, 1997)}, volume 1717 of {\em Lecture Notes in Math.}, pages
  93--191. Springer, Berlin, 1999.

\bibitem{martinelli2}
F.~Martinelli.
\newblock Relaxation times of {M}arkov chains in statistical mechanics and
  combinatorial structures.
\newblock In {\em Probability on discrete structures}, volume 110 of {\em
  Encyclopaedia Math. Sci.}, pages 175--262. Springer, Berlin, 2004.

\bibitem{Mar86}
K.~Marton.
\newblock A simple proof of the blowing-up lemma.
\newblock {\em IEEE Trans. Inform. Theory}, 32(3):445--446, 1986.

\bibitem{Mar96b}
K.~Marton.
\newblock Bounding {$\overline d$}-distance by informational divergence: a
  method to prove measure concentration.
\newblock {\em Ann. Probab.}, 24(2):857--866, 1996.

\bibitem{Mar96a}
K.~Marton.
\newblock A measure concentration inequality for contracting {M}arkov chains.
\newblock {\em Geom. Funct. Anal.}, 6(3):556--571, 1996.

\bibitem{Dualdiscrete}
N.Gozlan, P-M.Samson C.Roberto, and P.Tetali.
\newblock Kantorovich duality for marton's transport costs and applications.
\newblock 2014.

\bibitem{OV00}
F.~Otto and C.~Villani.
\newblock Generalization of an inequality by {T}alagrand and links with the
  logarithmic {S}obolev inequality.
\newblock {\em J. Funct. Anal.}, 173(2):361--400, 2000.

\bibitem{Saloff}
E.~Gin\'e G.R. Grimmett~L. Saloff-Coste.
\newblock {\em Lectures on Probability Theory and Statistics}.
\newblock Springer, 1997.

\bibitem{Sam00}
P.-M. Samson.
\newblock Concentration of measure inequalities for {M}arkov chains and
  {$\Phi$}-mixing processes.
\newblock {\em Ann. Probab.}, 28(1):416--461, 2000.

\bibitem{Sam03}
P.-M. Samson.
\newblock Concentration inequalities for convex functions on product spaces.
\newblock In {\em Stochastic inequalities and applications}, volume~56 of {\em
  Progr. Probab.}, pages 33--52. Birkh\"auser, Basel, 2003.

\bibitem{Sam07}
P.-M. Samson.
\newblock Infimum-convolution description of concentration properties of
  product probability measures, with applications.
\newblock {\em Ann. Inst. H. Poincar\'e Probab. Statist.}, 43(3):321--338,
  2007.

\bibitem{Villani}
C.~Villani.
\newblock {\em Optimal transport}, volume 338 of {\em Grundlehren der
  Mathematischen Wissenschaften [Fundamental Principles of Mathematical
  Sciences]}.
\newblock Springer-Verlag, Berlin, 2009.
\newblock Old and new.

\bibitem{Win13}
O.~Wintenberger.
\newblock Weak transport inequalities and applications to exponential and
  oracle inequalities.
\newblock Preprint, 2013.

\end{thebibliography}

\end{document}